\documentclass[11pt]{article}

\usepackage[T1]{fontenc}
\usepackage{lmodern}
\usepackage[margin=1in]{geometry}
\usepackage{microtype}
\usepackage{amsmath,amssymb,amsthm,mathtools}
\usepackage{aliascnt}
\usepackage{booktabs,tabularx,array}
\usepackage{graphicx}
\usepackage{enumitem}
\usepackage[round,authoryear]{natbib}
\usepackage{xcolor}
\usepackage[colorlinks=true,linkcolor=blue!55!black,citecolor=blue!55!black,urlcolor=blue!55!black]{hyperref}

\hypersetup{
  pdftitle={Universality of e-detectors for ARL control},
  pdfauthor={Professor Aaditya Ramdas}
}
\graphicspath{{universality_revision_assets/}}

\newcommand{\E}{\mathbb{E}}

\newcommand{\cF}{\mathcal{F}}
\newcommand{\cG}{\mathcal{G}}
\newcommand{\cP}{\mathcal{P}}

\newcommand{\1}{\mathbf{1}}
\newcommand{\N}{\mathbb{N}}
\newcommand{\Nzero}{\mathbb{N}_0}

\newcommand{\ARL}{\operatorname{ARL}}
\newcommand{\PFA}{\operatorname{PFA}}
\newcommand{\FAR}{\operatorname{FAR}}
\newcommand{\esssup}{\operatorname*{ess\,sup}}
\theoremstyle{plain}
\newtheorem{theorem}{Theorem}[section]
\newaliascnt{proposition}{theorem}
\newtheorem{proposition}[proposition]{Proposition}
\aliascntresetthe{proposition}
\newaliascnt{lemma}{theorem}
\newtheorem{lemma}[lemma]{Lemma}
\aliascntresetthe{lemma}
\newaliascnt{corollary}{theorem}
\newtheorem{corollary}[corollary]{Corollary}
\aliascntresetthe{corollary}

\theoremstyle{definition}
\newaliascnt{definition}{theorem}
\newtheorem{definition}[definition]{Definition}
\aliascntresetthe{definition}
\newaliascnt{example}{theorem}
\newtheorem{example}[example]{Example}
\aliascntresetthe{example}
\newaliascnt{openproblem}{theorem}

\aliascntresetthe{openproblem}

\theoremstyle{remark}
\newaliascnt{remark}{theorem}
\newtheorem{remark}[remark]{Remark}
\aliascntresetthe{remark}

\usepackage[nameinlink,noabbrev]{cleveref}
\crefname{equation}{equation}{equations}
\crefname{theorem}{theorem}{theorems}
\crefname{proposition}{proposition}{propositions}
\crefname{lemma}{lemma}{lemmas}
\crefname{corollary}{corollary}{corollaries}
\crefname{definition}{definition}{definitions}
\crefname{example}{example}{examples}
\crefname{openproblem}{open problem}{open problems}
\crefname{remark}{remark}{remarks}
\hypersetup{
  colorlinks=true,
  linkcolor=red,
  urlcolor=red,
  citecolor=red
}

\begin{document}

\title{Universality of \(e\)-detectors for ARL control}
\author{
Aaditya Ramdas\thanks{Department of Statistics, Stanford University. Email: \href{mailto:aramdas@stanford.edu}{\texttt{aramdas@stanford.edu}}}
}
\date{\today}
\maketitle

\begin{abstract}
An e-detector for a pre-change class $\mathcal P$ is a nonnegative process $M$ such that $\E_P[M_\tau] \leq \E_P[\tau]$ for all stopping times $\tau$ and all $P \in \mathcal P$. Thresholding e-detectors  controls the average run length (ARL): declaring a change at the first time $T_b$  when $M$ crosses $b$ ensures that $\inf_{P \in \mathcal P}\E_P[T] \geq b$. But e-detectors do substantially more than control the ARL; they also satisfy a \emph{optional-horizon inequality}:
\[
P(T_b\leq\sigma)\leq \E_P[\sigma]/b
\]
for every data-dependent stopping time (monitoring horizon) \(\sigma\) and $P\in \mathcal P$. In particular, every e-detector-based procedure obeys \(P(T\leq t)\leq t/b\) at each fixed $t$, thus avoiding early false alarms. Remarkably, the converse also holds: every stopping time $T$ that satisfies the optional-horizon inequality must in fact arise from thresholding an e-detector.
We also derive a universal representation of stopping times that satisfy (only) ARL control. These are represented by \emph{weak} e-detectors, that only require $\E_P[M_\tau] \leq \E_P[\tau]$ to hold at all threshold stopping times $T_b$.
Appendices present universal representations  for other (less common) change detection metrics.
\end{abstract}

\noindent\textbf{Keywords:} average run length; change detection; e-detector; e-process; false alarm; first-passage time; level-crossing time; optional stopping; sequential testing; universality.

\section{Introduction}\label{sec:introduction}

A useful way to understand a statistical formalism is to ask whether it is merely sufficient for validity or whether it is \emph{universal}: can every valid procedure be recovered within that formalism? Such representation results distinguish restrictions intrinsic to the error criterion from restrictions introduced by a construction method.

E-values give the simplest example. If \(R\in\{0,1\}\) is the realized decision of a level-\(\alpha\) test, then \(E=R/\alpha\) is an e-value, and the test rejects exactly when \(E\geq1/\alpha\). Randomized tests are covered by adjoining the external randomizer that produces the realized binary decision. Thus every hypothesis test can be represented by thresholding an e-value; see \citet{ramdaswang2025}.

The same phenomenon holds sequentially. A one-sided sequential test is identified with an alarm time \(T\) satisfying
\[
\sup_{P\in\cP}P(T<\infty)\leq\alpha.
\]
Then
\begin{equation}\label{eq:canonical-eprocess-intro}
E_n:=\frac1\alpha\1\{T\leq n\},\qquad n\geq0,
\end{equation}
is an e-process, and its first crossing of \(1/\alpha\) is exactly \(T\). Conversely, thresholding an e-process controls the probability of ever crossing. Hence every sequential test---and every change-detection rule controlling the probability of false alarm (PFA)---is recovered by thresholding an e-process. Proposition~7.8 of \citet{ramdaswang2025} gives this exact sequential-test representation.

Change detection is more often calibrated by average run length (ARL), though we consider other metrics later. Under a no-change law \(P\), the ARL of an alarm time \(T\) is \(\E_P[T]\), and a conventional target is
\begin{equation}\label{eq:intro-arl}
\inf_{P\in\cP}\E_P[T]\geq b.
\end{equation}
Unlike PFA control, \eqref{eq:intro-arl} is only a mean constraint. It permits eventual false alarm with probability one, and it allows false-alarm probability to be distributed very unevenly over time. This leads to the paper's motivating question:

\begin{quote}
Is there an object for ARL control that has an exact universality property analogous to those of e-values and e-processes?
\end{quote}

A natural candidate is the e-detector of \citet{shin2024}. A nonnegative adapted process \(M\), initialized at zero, is an e-detector when
\begin{equation}\label{eq:intro-strong}
\E_P[M_\tau]\leq\E_P[\tau]
\end{equation}
for every null law and every stopping horizon \(\tau\), with extended expectations; equivalently, it is enough to check bounded horizons. The above authors showed that thresholding an e-detector at \(b\) gives an alarm time with ARL at least \(b\). The original definition is also stable under  mixtures, which is crucial for the Shiryaev--Roberts-, CUSUM-, and mixture-style constructions their work.

The universality answer (to the earlier posed question) is subtle because \eqref{eq:intro-strong} encodes much more than ARL. If
\[
T_b(M):=\inf\{n\geq1:M_n\geq b\},
\]
then every integrable, possibly data-dependent stopping time \(\sigma\) satisfies the following \emph{optional horizon inequality}:
\begin{equation}\label{eq:intro-optional-tail}
P\bigl(T_b(M)\leq\sigma\bigr)\leq\frac{\E_P[\sigma]}{b}.
\end{equation}
In particular,
\begin{equation}\label{eq:intro-deterministic-tail}
P\bigl(T_b(M)\leq t\bigr)\leq\frac{t}{b}\qquad(t\in\N).
\end{equation}
Thus an e-detector-based procedure cannot obtain a large ARL by placing excessive false-alarm mass near the beginning and compensating with a very long right tail. This distributional guarantee is one of the paper's principal practical conclusions. 

It turns out that e-detectors precisely characterize all alarm times for which the optional horizon inequality holds. We prove that if some alarm time $T$ satisfies $P(T \leq \sigma) \leq \E_P[\sigma]/b$ for all stopping times $\sigma$, then $T$ must have been obtained by thresholding some e-detector at $b$, i.e.\ $T = T_b(M)$ for some e-detector $M$.

We also introduce the notion of a \emph{weak} e-detector, which is only required to satisfy~\eqref{eq:intro-strong} for a particular subset of stopping times: the level-crossings of $M$ itself, i.e.\ for stopping times of the form $T_b(M)$ only. Weak e-detectors do not satisfy the optional horizon inequality, but thresholding them does control the ARL. In fact, every stopping time $T$ that controls (only) the ARL can be recovered by thresholding a weak e-detector.

\paragraph{Clocked e-processes.}
A convenient way to organize and unify these ideas is through a \emph{clock} \(C\), which is a nondecreasing adapted process that records the accumulated false-alarm budget. A \emph{strong \(C\)-clocked e-process} satisfies
\[
\E_P[M_\tau]\leq\E_P[C_\tau]
\]
at every admissible horizon. The choices \(C_n\equiv1\) and \(C_n=n\) recover, respectively, e-processes and the original e-detectors. We prove the following general equivalence:
\begin{equation}\label{eq:intro-clock-equivalence}
P(T\leq\sigma)\leq\frac{\E_P[C_\sigma]}{b}\ \text{for all admissible }\sigma
\quad\Longleftrightarrow\quad
T=T_b(M)\ \text{for a strong \(C\)-clocked e-process}.
\end{equation}
The canonical witness is always \(M_n=b\1\{T\leq n\}\).
The optional-horizon theorem also leads to an exact characterization of null run-length distributions. For every strongly representable rule, the self-censoring horizon \(T\wedge m\) gives the necessary inequality
\begin{equation}\label{eq:intro-truncated-law}
bP(T\leq m)\leq \E_P[T\wedge m],\qquad m\in\N.
\end{equation}
When the filtration is the minimal one generated by the alarm time itself, these inequalities are also sufficient. Thus \eqref{eq:intro-truncated-law} completely characterizes which run-length laws can arise from a strong e-detector in the absence of side information. It is strictly stronger than the deterministic tail and ARL constraints taken together, yet it includes both geometric and deterministic run lengths. Under richer filtrations, the law alone is not enough: information revealed before the alarm can create adaptive horizons that violate optional validity.

For a lower bound on the expected clock at the alarm,
\[
\inf_{P\in\cP}\E_P[C_T]\geq b,
\]
the exact representing object is weaker: the clock inequality need only hold at the process's own level-crossing times. With the calendar clock \(C_n=n\), this yields the weak e-detector, because the above display reduces to ARL control. Weak e-detectors are sound at every threshold and  universal for ARL. We call the original e-detector \emph{strong} only to distinguish these two notions in this paper.

The weak theorem should be interpreted carefully. Its canonical certificate is a restatement of the target alarm event, and the weak class is not even closed under convex combinations. By contrast, the strong class is convex and has a genuinely checkable optimal-stopping certificate: for each null law \(P\), \(M_n-n\) is dominated by an integrable \(P\)-supermartingale with initial value zero. This Snell-envelope characterization explains why strong e-detectors support the mixture and aggregation operations used in practice, while weak e-detectors are best viewed as a language for formal completeness.
The main correspondences are summarized in \Cref{tab:universality}.

\begin{table}[t]
\centering
\small
\renewcommand{\arraystretch}{1.25}
\begin{tabularx}{\textwidth}{@{}>{\raggedright\arraybackslash}p{0.18\textwidth}>{\raggedright\arraybackslash}p{0.26\textwidth}>{\raggedright\arraybackslash}p{0.24\textwidth}X@{}}
\toprule
Setting & Validity condition & Canonical witness & Representing class \\
\midrule
Fixed test & \(P(R=1)\leq\alpha\) & \(E=\alpha^{-1}R\) & e-value \\
Sequential test / PFA & \(P(T<\infty)\leq\alpha\) & \(E_n=\alpha^{-1}\1\{T\leq n\}\) & e-process \\
Proper expected clock & \(\E[C_T]\geq b\) & \(M_n=b\1\{T\leq n\}\) & weak \(C\)-clocked process \\
Optional clock & \(P(T\leq\sigma)\leq\E[C_\sigma]/b\) & \(M_n=b\1\{T\leq n\}\) & strong \(C\)-clocked process \\
\bottomrule
\end{tabularx}
\caption{Four exact representation statements, shown for one null law to reduce notation. The expected-clock row requires a proper clock. ARL and strong e-detector validity are obtained by taking \(C_n=n\); e-process validity is obtained from the fixed clock \(C_n\equiv1\). The fixed clock has no nontrivial weak counterpart at scales \(b>1\), because \(C_T\equiv1\).}
\label{tab:universality}
\end{table}

Beyond the representation theorems, we develop a structural theory for the two classes. Strong e-detectors are convex and stable under independent enlargement, whereas weak e-detectors are nonconvex but admit a monotone normalization that the strong class does not preserve. Convexity permits an independent exponential clock to pin the actual ARL of a nondegenerate strong detector between two explicit bounds. A optimal-stopping value and its ratio form, the maximal strong scale, quantify the weak--strong gap. Snell envelopes characterize strong validity through both supermartingale and martingale domination, while the truncated-mean inequalities above exactly characterize strong run-length laws under the minimal alarm filtration. A time-inhomogeneous Gaussian Shewhart chart then supplies a genuine detection procedure with exact ARL that violates the strong optional-horizon bound: its first-observation power is larger than that of any strongly representable rule in the same Gaussian experiment, but its worst-case continuation delay is correspondingly larger.

The paper is organized as follows. \Cref{sec:background} fixes conventions and reviews related work. \Cref{sec:clocked} develops the clock abstraction that unifies e-processes and e-detectors. \Cref{sec:nonuniversal} shows why strong e-detectors are not universal for ARL alone and gives lottery and Gaussian-chart examples. \Cref{sec:weak} proves weak universality, studies closure under mixtures, and gives an exponential-clock calibration. \Cref{sec:eprocess-analogy} separates own-crossing from optional validity. \Cref{sec:strong-universality} characterizes the strong class, derives the run-length-law and Snell-envelope characterizations, and studies tail and hazard implications. \Cref{sec:discussion} addresses constructibility and statistical consequences. The references appear before a set of self-contained appendices extending the same representation program to other false-alarm metrics; a guide explains their purpose and organization.

\section{Background, notation, and related work}\label{sec:background}

\subsection{Filtered setup and stopping conventions}

Let \((\Omega,\cF,(\cF_n)_{n\geq0})\) be a filtered measurable space, and let \(\cP\) be a nonempty class of probability measures on \((\Omega,\cF)\). The laws in \(\cP\) represent possible no-change (or pre-change) distributions. No independence, identical-distribution, domination, or parametric assumptions are imposed.

We write \(\N=\{1,2,\ldots\}\) and \(\Nzero=\{0,1,2,\ldots\}\). Generic stopping horizons such as \(\tau\) and \(\sigma\) may take values in \(\Nzero\cup\{\infty\}\). An \emph{alarm time} or \emph{detection rule} \(T\), however, always takes values in \(\N\cup\{\infty\}\); in particular, an alarm cannot occur at time zero. This convention is essential for exact threshold representations because every level-crossing time below starts at time one.

For false-alarm events involving a possibly infinite horizon, \(\{T\leq\sigma\}\) means \(\{T<\infty,\,T\leq\sigma\}\). The distinction is immaterial when \(\sigma\) is almost surely finite. For a nonnegative process \(X\), define at a possibly infinite stopping time
\[
X_\tau:=\liminf_{N\to\infty}X_{\tau\wedge N}.
\]
We also use the killed stopped value
\begin{equation}\label{eq:killed}
X_\tau^\dagger:=\sum_{n=0}^{\infty}X_n\1\{\tau=n\},
\end{equation}
which equals zero on \(\{\tau=\infty\}\). Whenever \(\E_P[\tau]<\infty\), the stopping time is finite almost surely and the two conventions agree.

For an adapted nonnegative process \(M\) and \(c>0\), its level-crossing time is
\begin{equation}\label{eq:level-crossing}
T_c(M):=\inf\{n\geq1:M_n\geq c\},\qquad \inf\varnothing:=\infty.
\end{equation}
The term \emph{first-passage time} is equally standard; ``own level crossing'' emphasizes that the stopping rule is generated by the same process whose validity is being assessed.

We will repeatedly use the following clocked localization fact. It removes all integrability caveats from the validity definitions below: extended expectations are permitted, and it is enough to test bounded horizons.

\begin{lemma}[Clocked localization]\label{lem:localization}
Let \(M\) be nonnegative and adapted, and let \(C\) be a nonnegative adapted process that is pathwise nondecreasing. Define
\[
M_\tau:=\liminf_{N\to\infty}M_{\tau\wedge N},
\qquad
C_\tau:=\lim_{N\to\infty}C_{\tau\wedge N}.
\]
If
\begin{equation}\label{eq:clocked-localization-bounded}
\E_P[M_\tau]\leq\E_P[C_\tau]
\end{equation}
for every bounded stopping time \(\tau\), then \eqref{eq:clocked-localization-bounded} holds for every stopping time, possibly infinite. Hence validity over bounded, almost surely finite, or arbitrary stopping times is equivalent whenever the right-hand side is interpreted in \([0,\infty]\).
\end{lemma}

\begin{proof}
For an arbitrary stopping time \(\tau\), apply the bounded-horizon property to \(\tau\wedge N\). Fatou's lemma on the left and monotone convergence on the right give
\[
\E_P[M_\tau]
\leq\liminf_{N\to\infty}\E_P[M_{\tau\wedge N}]
\leq\lim_{N\to\infty}\E_P[C_{\tau\wedge N}]
=\E_P[C_\tau].
\]
\end{proof}


For a stopping rule \(T\), define
\begin{align}
\PFA_{\cP}(T)&:=\sup_{P\in\cP}P(T<\infty),\label{eq:pfa-def}\\
\ARL_{\cP}(T)&:=\inf_{P\in\cP}\E_P[T].\label{eq:arl-def}
\end{align}
PFA control at level \(\alpha\) requires \(\PFA_{\cP}(T)\leq\alpha\). ARL control at scale \(b\) requires \(\ARL_{\cP}(T)\geq b\); the usual parameterization is \(b=1/\alpha\), in which case the false-alarm rate \(\FAR=1/\ARL\) is at most \(\alpha\).

These criteria have fundamentally different geometry. PFA control reserves probability at infinity: under every null, the procedure must never alarm with probability at least \(1-\alpha\). ARL control permits eventual alarm with probability one. It constrains only the first moment of the alarm time and permits the distribution of false alarms to be highly front-loaded.

\subsection{E-values, e-processes, and e-detectors}

An e-value for \(\cP\) is a nonnegative random variable \(E\) satisfying \(\sup_{P\in\cP}\E_P[E]\leq1\). Markov's inequality yields a level-\(\alpha\) test by rejecting when \(E\geq1/\alpha\), while every realized binary level-\(\alpha\) test \(R\) gives the all-or-nothing e-value \(R/\alpha\).

An e-process is a nonnegative adapted process \(E\) satisfying
\begin{equation}\label{eq:eprocess-def}
\E_P[E_\tau]\leq1
\end{equation}
for every \(P\in\cP\) and every stopping time \(\tau\), with the infinite-horizon convention above. By \Cref{lem:localization} with the fixed clock \(C_n\equiv1\), it is equivalent to check bounded or almost surely finite stopping times. Ville's inequality turns threshold crossing at \(1/\alpha\) into a level-\(\alpha\) sequential test. Conversely, every level-\(\alpha\) sequential test is recovered by an all-or-nothing e-process; see Proposition~7.8 of \citet{ramdaswang2025}.
The e-detector of \citet[Definition~2.2]{shin2024} replaces the fixed unit budget of an e-process by a budget that grows linearly with elapsed time.

\begin{definition}[Strong e-detector; original definition]\label{def:strong}
A nonnegative adapted process \(M=(M_n)_{n\geq0}\), with \(M_0=0\), is a \emph{strong \(\cP\)-e-detector} if, for every \(P\in\cP\) and every stopping time \(\tau\),
\begin{equation}\label{eq:strong-def}
\E_P[M_\tau]\leq\E_P[\tau].
\end{equation}
\end{definition}

The inequality is interpreted in the extended sense and, by \Cref{lem:localization}, may equivalently be checked only at bounded horizons. The modifier \emph{strong} is used only to distinguish this original notion from \Cref{def:weak}.
Thresholding a strong e-detector at \(b\) gives ARL at least \(b\), as established by \citet[Theorem~2.4]{shin2024}. The class is convex: averages and, more generally,  mixtures of strong e-detectors remain strong \citep[Proposition~2.3]{shin2024}. 

\subsection{Related work}

ARL has been a central false-alarm metric since the early control-chart and quickest-detection literature. Shewhart's control chart \citep{shewhart1931}, Page's CUSUM \citep{page1954}, and the Shiryaev and Shiryaev--Roberts procedures \citep{shiryaev1963,roberts1966} are foundational. The minimax formulations of \citet{lorden1971} and \citet{pollak1985}, together with the exact Lorden optimality result of \citet{moustakides1986} and its decision-theoretic development by \citet{ritov1990}, form the classical quickest-detection backbone. Optimality properties of the Shiryaev--Roberts procedure are studied by \citet{pollaktartakovsky2009}; head-started and quasi-stationary variants and their higher-order behavior are developed by \citet{polunchenkotartakovsky2010,tartakovskypollakpolunchenko2012}. Standard monographs and surveys include \citet{siegmund1985,bassevillenikiforov1993,poorhadjiliadis2008,tartakovsky2014,xieetal2021}.

The run-length distribution has long been studied in addition to its mean. \citet{brooksevans1972} compute exact probabilities and quantiles for CUSUM run lengths, and \citet{pollaktartakovsky2009exit} establish asymptotic exponentiality for broad first-exit models. \citet{mei2008} emphasizes that a large ARL can hide substantial early false-alarm risk. Fast-initial-response and head-start designs intentionally trade later in-control behavior for quicker startup response \citep{lucascrosier1982}; their false-alarm probability functions can differ markedly even at similar ARL values \citep{nishinanishiyuki2003}. \citet{pergamenchtchikov2018} develops local unconditional and conditional false-alarm classes over moving windows. The equilibrium-distribution and new-better-than-used-in-expectation (NBUE) literature provides a complementary reliability-theoretic language for comparing a lifetime law with its length-biased residual-life law \citep{marshallolkin2007,nairsankaranpreeth2012}; \Cref{subsec:run-length-laws} identifies an exact connection to strong e-detector run lengths. These strands motivate our focus on the lower tail and on optional selection of the monitoring horizon.

Much of the classical optimality theory assumes specified pre- and post-change laws. Composite, adaptive, and robust variants use generalized likelihood ratios, mixtures, or least-favorable laws; representative contributions include \citet{lai1995,mei2006,unnikrishnanetal2011,xiesiegmund2013,molloyford2017}. Recent work develops nonparametric reductions using backward confidence sequences \citep{shekharramdas2023} and information-theoretic ARL--delay bounds for composite pre-change classes \citep{ramramdas2026}. These papers concern construction and efficiency. Our question is orthogonal: once a null-validity criterion is fixed, which stopping rules admit an exact threshold representation?

E-values and test martingales trace back at least to \citet{ville1939}. Modern methods in anytime-valid inference and game-theoretic statistics are surveyed by \citet{ramdasetal2023} and treated systematically by \citet{ramdaswang2025}. E-processes can be strictly more general than a single nonnegative supermartingale under composite nulls \citep{rufetal2023}. Nevertheless, Snell envelopes and martingale domination provide powerful structural descriptions of admissible anytime-valid procedures \citep{ramdasruflarssonkoolen2022}; we use the same  technique to characterize strong e-detectors.

\citet{shin2024} introduced e-detectors for nonparametric sequential change detection with finite-sample ARL guarantees. Their constructions aggregate e-processes started at candidate changepoints and recover Shiryaev--Roberts- and CUSUM-style statistics. The present work asks what validity condition these objects encode exactly, and why that condition is stronger and more construction-friendly than a bare ARL constraint.

\section{Clocked e-processes: a common representation theorem}\label{sec:clocked}

E-processes and e-detectors differ in the budget against which stopped expectation is compared. A \emph{clock} is a nonnegative adapted process \(C=(C_n)_{n\geq0}\) that is pathwise nondecreasing. For a stopping time \(T\), set
\begin{equation}\label{eq:clock-stopped}
C_T:=\lim_{N\to\infty}C_{T\wedge N}\in[0,\infty].
\end{equation}
The terminology in this section is proposed, but the two principal examples are established objects.

\begin{definition}[Strong clocked e-process]\label{def:clocked-eprocess}
A nonnegative adapted process \(M\) is a \emph{strong \(C\)-clocked \(\cP\)-e-process} if
\begin{equation}\label{eq:clocked-strong}
\E_P[M_\tau]\leq\E_P[C_\tau]
\end{equation}
for every \(P\in\cP\) and every stopping time \(\tau\). By \Cref{lem:localization}, bounded stopping times suffice.
\end{definition}

For \(C_n\equiv1\), \eqref{eq:clocked-strong} is exactly the e-process property by \Cref{lem:localization}. For \(C_n=n\), it is exactly the strong e-detector property. More generally, \(C_n=\sum_{j=1}^n a_j\) for predictable \(a_j\geq0\) measures accumulated exposure instead of calendar time.

\begin{definition}[Optional-clock false-alarm control]\label{def:optional-clock}
Fix \(b>0\). An alarm time \(T\) has \emph{optional-clock false-alarm control at scale \(b\)} if
\begin{equation}\label{eq:optional-clock-control}
P(T\leq\sigma)\leq\frac{\E_P[C_\sigma]}{b}
\end{equation}
for every \(P\in\cP\) and every stopping time \(\sigma\), with the inequality interpreted in the extended sense.
\end{definition}

\begin{theorem}[Optional-clock universality]\label{thm:clocked-strong-universality}
Fix \(b>0\) and an alarm time \(T\). The following are equivalent.
\begin{enumerate}[label=\textup{(\roman*)},leftmargin=2.4em]
\item \(T\) satisfies optional-clock false-alarm control \eqref{eq:optional-clock-control}.
\item There exists a strong \(C\)-clocked \(\cP\)-e-process \(M\) such that \(T=T_b(M)\).
\end{enumerate}
Under \textup{(i)}, the canonical witness is
\begin{equation}\label{eq:canonical-clocked}
M_n^{T,b}:=b\1\{T\leq n\}.
\end{equation}
\end{theorem}

\begin{proof}
Suppose \textup{(ii)} holds. Fix \(P\) and a stopping horizon \(\sigma\), and put \(\rho=T_b(M)\wedge\sigma\). Then \(C_\rho\leq C_\sigma\), while \(M_\rho\geq b\) on \(\{T_b(M)\leq\sigma\}\). Hence
\[
bP(T_b(M)\leq\sigma)
\leq\E_P[M_\rho]
\leq\E_P[C_\rho]
\leq\E_P[C_\sigma].
\]
Conversely, under \textup{(i)}, the process in \eqref{eq:canonical-clocked} crosses \(b\) exactly at \(T\), and for every stopping time \(\tau\),
\[
\E_P[M_\tau^{T,b}]
=bP(T\leq\tau)
\leq\E_P[C_\tau].
\]
\end{proof}

The optional-clock theorem has the following specializations:
\begin{enumerate}[label=\textup{(\alph*)},leftmargin=2.2em]
\item If \(C_n\equiv1\), then threshold crossing obeys \(P(T_b<\infty)\leq1/b\), and every PFA-controlled alarm time has an e-process representation.
\item If \(C_n=n\), then threshold crossing obeys optional-horizon linear false-alarm control, \(P(T_b\leq\sigma)\leq\E[\sigma]/b\).
\item If \(C_n=\sum_{j\leq n}a_j\), then the false-alarm budget is proportional to expected accumulated exposure.
\end{enumerate}


\begin{example}[An exposure clock]\label{ex:exposure-clock}
Suppose observation opportunities occur in calendar time, but the detector is updated only when a predictable indicator \(a_n\in\{0,1\}\) says that a sample is actually collected. The clock
\[
C_n:=\sum_{j=1}^n a_j
\]
counts samples rather than elapsed periods. Optional-clock control then gives
\[
P(T\leq\sigma)\leq \frac{\E[C_\sigma]}{b},
\]
so the false-alarm budget is charged per expected sample collected, even when the sampling schedule is adaptive. More generally, predictable costs \(a_j\geq0\) yield cost-weighted monitoring. This is useful for intermittent sensors, batched inspections, or settings in which observations have unequal acquisition costs.
\end{example}

There is a parallel weak theorem. Assume now that \(C_0=0\) and
\begin{equation}\label{eq:proper-clock}
C_n\uparrow\infty\quad P\text{-almost surely for every }P\in\cP.
\end{equation}
Call such a clock \emph{proper}.

\begin{definition}[Weak clocked e-process]\label{def:weak-clocked}
For a proper clock \(C\), a nonnegative adapted process \(M\), with \(M_0=0\), is a \emph{weak \(C\)-clocked \(\cP\)-e-process} if, for every \(P\in\cP\) and every \(c>0\),
\begin{equation}\label{eq:weak-clocked}
\E_P\bigl[M_{T_c(M)}^\dagger\bigr]
\leq
\E_P[C_{T_c(M)}].
\end{equation}
\end{definition}

\begin{theorem}[Expected-clock universality]\label{thm:weak-clocked-universality}
Let \(C\) be proper. Fix \(b>0\) and an alarm time \(T\). The following are equivalent.
\begin{enumerate}[label=\textup{(\roman*)},leftmargin=2.4em]
\item Clocked ARL control holds:
\begin{equation}\label{eq:expected-clock-control}
\inf_{P\in\cP}\E_P[C_T]\geq b.
\end{equation}
\item There exists a weak \(C\)-clocked \(\cP\)-e-process \(M\) such that \(T=T_b(M)\).
\end{enumerate}
Under \textup{(i)}, one may again take \(M_n=b\1\{T\leq n\}\).
\end{theorem}

\begin{proof}
Suppose \textup{(ii)} holds. If \(P(T=\infty)>0\), properness implies \(\E_P[C_T]=\infty\). Otherwise weak validity at level \(b\) gives
\[
b\leq\E_P[M_T^\dagger]\leq\E_P[C_T].
\]
Conversely, set \(M_n=b\1\{T\leq n\}\). At each level \(0<c\leq b\), its crossing time is \(T\), and
\[
\E_P[M_T^\dagger]\leq b\leq\E_P[C_T].
\]
At levels \(c>b\), it never crosses. Thus \(M\) is weakly \(C\)-clocked and represents \(T\).
\end{proof}

With \(C_n=n\), \Cref{thm:weak-clocked-universality} is precisely ARL universality. The remainder of the main text specializes these clocked results, identifies strict separations, and studies which version has useful closure and verification properties.

\section{Strong e-detectors are not universal for ARL}\label{sec:nonuniversal}

The failure of strong universality is already visible from a consequence of \eqref{eq:strong-def} that is stronger than mean control.

\begin{proposition}[Optional-horizon false-alarm inequality]\label{prop:strong-optional}
Let \(M\) be a strong \(\cP\)-e-detector, fix \(b>0\), and let \(T_b=T_b(M)\). Then, for every \(P\in\cP\) and every stopping time \(\sigma\),
\begin{equation}\label{eq:optional-bound}
P(T_b\leq\sigma)\leq \frac{\E_P[\sigma]}{b}.
\end{equation}
\end{proposition}

\begin{proof}
Set \(\rho:=T_b\wedge\sigma\). On \(\{T_b\leq\sigma\}\), one has \(\rho=T_b<\infty\), and hence \(M_\rho\geq b\). Therefore
\[
bP(T_b\leq\sigma)\leq \E_P[M_\rho]
\leq \E_P[\rho]
\leq \E_P[\sigma],
\]
where the middle inequality is the strong e-detector property.
\end{proof}

Taking \(\sigma\equiv t\) gives the following deterministic tail bound.

\begin{corollary}[Linear deterministic-horizon bound]\label{cor:linear-tail}
Under the assumptions of \Cref{prop:strong-optional}, for every integer \(t\geq1\),
\begin{equation}\label{eq:linear-tail}
P(T_b\leq t)\leq \frac{t}{b}.
\end{equation}
In particular, \(P(T_b=1)\leq1/b\).
\end{corollary}

The gap between ordinary ARL and optional-horizon control can be summarized by one number.

\begin{definition}[Maximal strong scale]\label{def:strong-scale}
For a null law \(P\), define
\begin{equation}\label{eq:strong-scale}
b_P^*(T)
:=
\inf_{\substack{\sigma\text{ a stopping time}\\ \E_P[\sigma]<\infty,\;P(T\leq\sigma)>0}}
\frac{\E_P[\sigma]}{P(T\leq\sigma)},
\end{equation}
with the infimum of the empty set interpreted as \(\infty\). For a composite null class, set
\[
b_{\cP}^*(T):=\inf_{P\in\cP}b_P^*(T).
\]
We call this the \emph{maximal strong scale} of \(T\).
\end{definition}

\begin{proposition}[Meaning of the strong scale]\label{prop:strong-scale}
An alarm time \(T\) satisfies optional-horizon linear false-alarm control at scale \(b\) over \(\cP\) if and only if
\[
b\leq b_{\cP}^*(T).
\]
Consequently, after the representation theorem in \Cref{thm:strong-universality}, \(b_{\cP}^*(T)\) is exactly the largest scale at which \(T\) is representable by a strong e-detector. If \(T\) is integrable under \(P\), then
\begin{equation}\label{eq:strong-scale-vs-arl}
b_P^*(T)\leq \E_P[T],
\end{equation}
so the ratio \(\E_P[T]/b_P^*(T)\) quantifies the weak--strong gap for that law and filtration.
\end{proposition}

\begin{proof}
The optional-horizon inequality is equivalent to
\(
b\leq \E_P[\sigma]/P(T\leq\sigma)
\)
for every horizon appearing in \eqref{eq:strong-scale}; taking the infimum over \(\sigma\) and then over \(P\) proves the first claim. If \(T\) is integrable, the choice \(\sigma=T\) gives \eqref{eq:strong-scale-vs-arl}.
\end{proof}

ARL control imposes no comparable bound. The following elementary example is useful because every quantity is explicit.

\begin{example}[An early-alarm lottery]\label{ex:early-lottery}
Let \(b=4\), let \(Y\sim\mathrm{Bernoulli}(1/2)\), reveal \(Y\) at time one, and define
\[
T:=
\begin{cases}
1,&Y=1,\\
7,&Y=0.
\end{cases}
\]
Then \(T\) is a stopping time and
\[
\E[T]=\frac12\cdot1+\frac12\cdot7=4=b.
\]
Thus \(T\) has exact ARL \(b\). Yet \(P(T=1)=1/2>1/4=1/b\), so \Cref{cor:linear-tail} rules out representing \(T\) as the level-\(b\) crossing time of any strong e-detector.
\end{example}

\begin{corollary}[Failure of strong universality for ARL]\label{cor:no-strong-univ}
For every \(b>1\), there exist a filtered probability space, a singleton null class \(\cP\), and an alarm time \(T\) with \(\ARL_{\cP}(T)\geq b\) that is not the level-\(b\) crossing time of any strong \(\cP\)-e-detector.
\end{corollary}


\section{Weak e-detectors and exact universality for ARL}\label{sec:weak}

For threshold-based ARL control, the defining inequality of a strong e-detector is applied at one special stopping time: the process's own threshold crossing. This motivates the following weakening.

\begin{definition}[Weak e-detector]\label{def:weak}
A nonnegative adapted process \(M=(M_n)_{n\geq0}\), with \(M_0=0\), is a \emph{weak \(\cP\)-e-detector} if, for every \(P\in\cP\) and every level \(c>0\),
\begin{equation}\label{eq:weak-def}
\E_P\bigl[M_{T_c(M)}^\dagger\bigr]
\leq
\E_P\bigl[T_c(M)\bigr].
\end{equation}
\end{definition}

Thus validity is required only at the process's own level-crossing times, at every level.
Clearly, every strong \(\cP\)-e-detector is a weak \(\cP\)-e-detector.
The weak definition retains exactly the argument needed for ARL control: thresholding a weak e-detector at $c$ results in a stopping time with ARL at least $c$.

\begin{theorem}[ARL control for weak e-detectors.]\label{thm:weak-soundness}
If \(M\) is a weak \(\cP\)-e-detector, then for every \(c>0\),
\begin{equation}\label{eq:weak-soundness}
\inf_{P\in\cP}\E_P[T_c(M)]\geq c.
\end{equation}
\end{theorem}

\begin{proof}
Fix \(P\in\cP\). If \(\E_P[T_c(M)]=\infty\), there is nothing to prove. Otherwise, \(T_c(M)<\infty\) almost surely and \(M_{T_c(M)}\geq c\). Hence
\[
c\leq \E_P[M_{T_c(M)}]
\leq \E_P[T_c(M)].
\]
Taking the infimum over \(P\) proves the claim.
\end{proof}

Weakness is strict. The early-alarm lottery from \Cref{ex:early-lottery} already supplies a transparent witness.

\begin{example}[A weak e-detector that is not strong]\label{ex:weak-not-strong}
For the stopping time \(T\) in \Cref{ex:early-lottery}, define
\begin{equation}\label{eq:lottery-canonical}
M_n:=4\1\{T\leq n\}.
\end{equation}
For every \(0<c\leq4\), the crossing time \(T_c(M)\) equals \(T\), and
\[
\E[M_{T_c(M)}]=4=\E[T_c(M)].
\]
For \(c>4\), the process never crosses and the weak inequality is automatic. Thus \(M\) is weak. It is not strong, since at the deterministic stopping time \(\tau\equiv1\),
\[
\E[M_1]=4P(T=1)=2>1=\E[\tau].
\]
\end{example}

\subsection{Closure under mixtures and the limits of weak constructibility}

The two detector classes have sharply different closure properties.

\begin{proposition}[Strong convexity and weak nonconvexity]\label{prop:convexity-separation}
\begin{enumerate}[label=\textup{(\alph*)},leftmargin=2.2em]
\item Strong e-detectors are closed under fixed mixtures.
\item The class of weak e-detectors need not be convex, even for a singleton null law.
\end{enumerate}
\end{proposition}

\begin{proof}
For \textup{(a)}, if \(M^a\) is strong for each \(a\) and \(\mu\) is a probability measure, Tonelli's theorem gives, for every admissible \(\tau\),
\[
\E_P\!\left[\int M_\tau^a\,d\mu(a)\right]
=\int\E_P[M_\tau^a] \,d\mu(a)
\leq\E_P[\tau].
\]
This is the mixture property of \citet[Proposition~2.3]{shin2024}.

For \textup{(b)}, let \(Y\sim\operatorname{Bernoulli}(1/2)\) be revealed at time one. Define
\[
T=\begin{cases}1,&Y=1,\\7,&Y=0,\end{cases}
\qquad
T'=\begin{cases}1,&Y=0,\\7,&Y=1.\end{cases}
\]
Both have mean four. Define
\[
M_n^{(1)}=4\1\{T\leq n\},
\qquad
M_n^{(2)}=4\1\{T'\leq n\}.
\]
For either process and every \(0<c\leq4\), the level-\(c\) crossing is the corresponding alarm time and the stopped value has expectation four, equal to the expected crossing time. For \(c>4\), the process never crosses. Thus both processes are weak e-detectors, by direct verification. Their average is deterministic:
\[
\frac{M_n^{(1)}+M_n^{(2)}}2
=\begin{cases}
2,&1\leq n\leq6,\\
4,&n\geq7.
\end{cases}
\]
At level \(c=2\), its crossing time is one and its stopped value is two, so
\[
\E\!\left[\left(\frac{M^{(1)}+M^{(2)}}2\right)_{T_2}\right]=2>1=\E[T_2].
\]
The average is not weak.
\end{proof}

There is a second, complementary asymmetry.

\begin{proposition}[Monotone normalization and running maxima]\label{prop:running-maxima}
\begin{enumerate}[label=\textup{(\alph*)},leftmargin=2.2em]
\item Every weak e-detector \(M\) may be replaced by the nondecreasing process
\[
\overline M_n:=\max_{0\leq j\leq n}M_j
\]
without changing any level-crossing time or any stopped value at a first crossing. In particular, \(\overline M\) is weak whenever \(M\) is weak.
\item Strong e-detectors are not closed under running maxima.
\end{enumerate}
\end{proposition}

\begin{proof}
For \textup{(a)}, fix \(c>0\). The first time that the running maximum reaches \(c\) is exactly the first time that \(M\) reaches \(c\). If this time is finite, every earlier value is below \(c\), whereas the crossing value is at least \(c\); hence \(\overline M_{T_c(M)}=M_{T_c(M)}\). Weak validity is therefore unchanged.

For \textup{(b)}, let \(K>2\), let \(B\) be an event with \(P(B)=2/K\), take \(\cF_1\) trivial and \(\cF_2=\sigma(B)\), and define
\[
M_0=0,\qquad M_1=1,\qquad M_2=K\1_B,\qquad M_n=0\quad(n\geq3).
\]
If a stopping time \(\tau\) is not identically zero or one, triviality of \(\cF_1\) forces \(\tau\geq2\) almost surely, and then
\[
\E[M_\tau]\leq KP(B)=2\leq\E[\tau].
\]
The cases \(\tau\equiv0\) and \(\tau\equiv1\) are immediate, so \(M\) is strong. At deterministic time two, however,
\[
\E[\overline M_2]
=K P(B)+P(B^c)
=3-\frac2K
>2.
\]
Thus \(\overline M\) is not strong.
\end{proof}

The failure is not cosmetic. Taking mixtures over betting parameters is central e-detector construction tool. Strong validity survives this operation automatically; weak validity does not. Weak e-detectors therefore provide an exact representation of ARL-valid stopping rules, but not a comparably robust recipe for building them.

Convexity also yields a useful way to make the actual ARL finite and quantitatively close to the threshold while retaining a data-driven detector. Randomization enlarges the filtration, so we first record the stability statement that makes this operation legitimate.

\begin{lemma}[Independent enlargement]\label{lem:independent-enlargement}
Let \(M\) be a strong \(C\)-clocked \(\cP\)-e-process on \((\Omega,\cF,(\cF_n))\). Let \((\Omega',\cF',(\cF_n'),Q)\) carry independent auxiliary randomness, and equip the product space with
\[
\widetilde\cF_n:=\cF_n\otimes\cF_n',
\qquad
\widetilde\cP:=\{P\otimes Q:P\in\cP\}.
\]
Viewing \(M\) and \(C\) as processes on the product space that do not depend on \(\omega'\), \(M\) remains a strong \(C\)-clocked \(\widetilde\cP\)-e-process.
\end{lemma}

\begin{proof}
Let \(\tau\) be a stopping time for the product filtration. For each fixed \(\omega'\), the section \(\tau_{\omega'}(\omega):=\tau(\omega,\omega')\) is an \((\cF_n)\)-stopping time, because sections of \(\{\tau\leq n\}\in\cF_n\otimes\cF_n'\) belong to \(\cF_n\). Strong clocked validity holds fiberwise, and Tonelli's theorem gives
\[
\E_{P\otimes Q}[M_\tau]
=\int \E_P[M_{\tau_{\omega'}}]Q(d\omega')
\leq\int \E_P[C_{\tau_{\omega'}}]Q(d\omega')
=\E_{P\otimes Q}[C_\tau].
\]
\end{proof}

\begin{proposition}[Exponential-clock ARL pinning]\label{prop:exp-clock-pinning}
Let \(M\) be a strong \(\cP\)-e-detector on \((\Omega,\cF,(\cF_n))\). On an auxiliary filtered space, let \(U_1,U_2,\ldots\) be i.i.d. \(\operatorname{Exp}(1)\) under \(Q\), revealed sequentially, and use the product filtration and product null class
\[
\widetilde\cP:=\{P\otimes Q:P\in\cP\}.
\]
Define
\[
B_n:=\sum_{j=1}^n U_j,
\qquad
\widetilde M_n^{(\rho)}:=(1-\rho)M_n+\rho B_n,
\qquad \rho\in(0,1].
\]
Then the pulled-back process \(M\), the auxiliary clock \(B\), and \(\widetilde M^{(\rho)}\) are strong \(\widetilde\cP\)-e-detectors. Moreover, for every \(\gamma>0\) and every \(\widetilde P=P\otimes Q\in\widetilde\cP\),
\begin{equation}\label{eq:exp-clock-pinning}
\gamma
\leq \E_{\widetilde P}\!\left[T_\gamma\!\left(\widetilde M^{(\rho)}\right)\right]
\leq \frac{\gamma}{\rho}+1.
\end{equation}
In particular, taking \(\rho=(1+\eta)^{-1}\) gives
\[
\gamma
\leq \E_{\widetilde P}\!\left[T_\gamma\!\left(\widetilde M^{(\rho)}\right)\right]
\leq (1+\eta)\gamma+1.
\]
\end{proposition}

\begin{proof}
By \Cref{lem:independent-enlargement}, \(M\) remains strong on the product filtration. It is enough to verify validity of \(B\) at bounded stopping times. If \(\tau\leq N\), then \(\{\tau\geq n\}\in\widetilde\cF_{n-1}\), so it is independent of the fresh increment \(U_n\). Tonelli's theorem yields
\[
\E_{\widetilde P}[B_\tau]
=\sum_{n=1}^N\E_{\widetilde P}[U_n\1\{\tau\geq n\}]
=\sum_{n=1}^N\widetilde P(\tau\geq n)
=\E_{\widetilde P}[\tau].
\]
Thus \(B\) is strong by \Cref{lem:localization}, and \(\widetilde M^{(\rho)}\) is strong by \Cref{prop:convexity-separation}. The lower bound in \eqref{eq:exp-clock-pinning} follows from strong-e-detector soundness. Since \(M\geq0\),
\[
T_\gamma\!\left(\widetilde M^{(\rho)}\right)
\leq
S_{\gamma/\rho}
:=\inf\{n\geq1:B_n\geq\gamma/\rho\}.
\]
The partial sums \(B_n\) are the arrival times of a unit-rate Poisson process, so \(S_x\) has the same law as \(1+\operatorname{Poisson}(x)\), and hence \(\E[S_x]=x+1\). This proves the upper bound.
\end{proof}


Since \(B_n\uparrow\infty\) almost surely,
\[
T_\gamma\!\left(\widetilde M^{(\rho)}\right)
\leq S_{\gamma/\rho}<\infty
\qquad\widetilde P\text{-almost surely},
\]
and more precisely
\[
\widetilde P\!\left(T_\gamma\!\left(\widetilde M^{(\rho)}\right)>n\right)
\leq Q(B_n<\gamma/\rho).
\]
Thus the auxiliary clock deliberately forces an eventual null alarm, even on paths where the data-driven component would never cross. 

\subsection{A universality result for ARL control}

We now obtain the exact universality theorem for ordinary ARL control.

\begin{theorem}[Universality of weak e-detectors]\label{thm:weak-universality}
Fix \(b>0\) and a stopping time \(T\). The following statements are equivalent.

\begin{enumerate}[label=\textup{(\roman*)},leftmargin=2.4em]
\item \(\inf_{P\in\cP}\E_P[T]\geq b\).
\item There exists a weak \(\cP\)-e-detector \(M\) such that \(T=T_b(M)\).
\end{enumerate}

When \textup{(i)} holds, one may take the canonical weak e-detector
\begin{equation}\label{eq:canonical-weak}
M_n^{T,b}:=b\1\{T\leq n\},\qquad n\geq0.
\end{equation}
\end{theorem}

\begin{proof}
The implication \textup{(ii)}\(\Rightarrow\)\textup{(i)} is \Cref{thm:weak-soundness}.

For the converse, assume \textup{(i)} and define \(M^{T,b}\) by \eqref{eq:canonical-weak}. Adaptedness follows from the stopping-time property of \(T\). Pathwise, \(M_n^{T,b}=0\) before \(T\) and equals \(b\) from \(T\) onward, so its level-\(b\) crossing time is exactly \(T\).

It remains to verify weak validity at an arbitrary level \(c>0\). If \(c>b\), the process never reaches \(c\), so \(T_c(M^{T,b})=\infty\) and \eqref{eq:weak-def} is automatic. If \(0<c\leq b\), then \(T_c(M^{T,b})=T\). Fix \(P\in\cP\). If \(\E_P[T]=\infty\), the desired inequality is again automatic. If \(\E_P[T]<\infty\), then \(T<\infty\) almost surely and
\[
\E_P\bigl[M_{T_c(M^{T,b})}^{T,b}\bigr]
=b
\leq \E_P[T]
=\E_P\bigl[T_c(M^{T,b})\bigr].
\]
Thus \(M^{T,b}\) is weak.
\end{proof}

Several points about \Cref{thm:weak-universality} are worth recording.

\begin{remark}[Level-free validity, scale-specific representation]\label{rem:level-free}
The weak definition is level-free: one process must certify ARL at every threshold it crosses. The canonical witness for a target scale \(b\), however, is intentionally simple. It represents \(T\) at level \(b\), represents the same \(T\) at all lower levels, and never crosses a higher level. The theorem is therefore a universality statement at each target scale; it does not claim that an arbitrary family of stopping rules indexed by \(b\) can be represented by one common process.
\end{remark}

\begin{remark}[Overshoots]\label{rem:overshoot}
The weak condition uses the actual stopped value \(M_{T_c(M)}\), not merely the threshold \(c\). This accounts automatically for overshoot. The canonical indicator witness has no overshoot, but practical weak or strong e-detectors may.
\end{remark}

\begin{corollary}[Early-tail restrictions preclude ARL universality]\label{cor:anytime-precludes}
For every \(b>1\) and \(\gamma<1\), there exist a filtered probability space and an alarm time \(T\) such that
\[
\E[T]\geq b
\qquad\text{but}\qquad
P(T=1)>\gamma.
\]
Consequently, no representation principle that imposes the same nontrivial upper bound \(P(T=1)\leq\gamma\) on every filtered probability space can be universal for ARL control.
\end{corollary}

\begin{proof}
Take a space supporting a Bernoulli variable revealed at time one. Choose \(p\in(\gamma,1)\), let \(T=1\) with probability \(p\), and let \(T=L\) otherwise for an integer \(L\) large enough that \(\E[T]\geq b\).
\end{proof}

The corollary rules out any universal class that retains a nontrivial common upper bound on first-step false-alarm probability. It does not make the weak definition logically minimal. At a fixed target scale \(b\), imposing the own-crossing inequality only at level \(b\) would already suffice for the canonical representation. Requiring it at every level instead gives one process that is sound at every threshold, as in \Cref{thm:weak-soundness}, without sacrificing scale-by-scale universality.

\section{Pros and cons of weak versus strong detectors}

\begin{example}[A sharp fast-start Gaussian tradeoff]\label{ex:fast-start-shewhart}
Suppose that under the no-change law \(X_1,X_2,\ldots\) are i.i.d. \(N(0,1)\). Fix \(b\geq2\) and \(p\in(1/b,1)\). At time one, raise an alarm when
\[
X_1\geq a:=\Phi^{-1}(1-p).
\]
Conditional on no alarm at time one, use the more conservative threshold
\[
c:=\Phi^{-1}(1-q),
\qquad
q:=\frac{1-p}{b-1},
\]
and raise an alarm at the first \(n\geq2\) with \(X_n\geq c\). The resulting alarm time $T$ is a time-inhomogeneous Shewhart chart. Its null run length satisfies
\[
P_\infty(T=1)=p,
\qquad
P_\infty(T\geq n)=(1-p)(1-q)^{n-2}\quad(n\geq2),
\]
and therefore
\[
\E_\infty[T]
=1+\frac{1-p}{q}
=b.
\]
Because \(p>1/b\), the chart violates \Cref{cor:linear-tail} and has no strong e-detector representation at level \(b\), although it has an exact weak representation. In fact, its maximal strong scale can be computed exactly in the full data filtration. For any integrable stopping time \(\sigma\), the conditional null hazard is \(p\) at time one and \(q\) thereafter, so
\begin{align*}
P_\infty(T\leq\sigma)
&=pP_\infty(\sigma\geq1)
  +q\sum_{n=2}^{\infty}P_\infty(\sigma\geq n,\,T\geq n)\\
&\leq p\sum_{n=1}^{\infty}P_\infty(\sigma\geq n)
=p\E_\infty[\sigma],
\end{align*}
because \(q<p\). Hence \(b_{P_\infty}^*(T)\geq1/p\), while \(\sigma\equiv1\) attains equality. Therefore
\begin{equation}\label{eq:fast-start-strong-scale}
b_{P_\infty}^*(T)=\frac1p.
\end{equation}
Thus a chart with ARL \(b\) can be strongly representable only at the much smaller scale \(1/p\).

For an immediate shift to \(N(\mu,1)\), \(\mu>0\), define
\[
\beta_\mu(r)
:=1-\Phi\!\left(\Phi^{-1}(1-r)-\mu\right),
\]
the power of the most powerful size-\(r\) test of \(N(0,1)\) against \(N(\mu,1)\). Every strongly representable rule at scale \(b\) obeys \(P_\infty(T=1)\leq1/b\). Hence, when the time-one information consists of \(X_1\) and an independent randomizer, the Neyman--Pearson lemma \citep{neymanpearson1933} gives
\begin{equation}\label{eq:strong-first-step-power}
P_{1,\mu}(T=1)\leq \beta_\mu(1/b).
\end{equation}
The fast-start chart has first-observation power \(\beta_\mu(p)>\beta_\mu(1/b)\). The stationary Shewhart chart with null hazard \(1/b\) attains the upper bound in \eqref{eq:strong-first-step-power}, so \(\beta_\mu(1/b)\) is exactly the maximum first-observation power over the strong class in this experiment.

The same calculation identifies the price under Lorden's criterion. Let \(P_{\nu,\mu}\) denote the law with a change to \(N(\mu,1)\) at time \(\nu\), and write
\[
\mathcal D_{\mathrm L}(T)
:=\sup_{\nu\geq1}\operatorname*{ess\,sup}
\E_{\nu,\mu}\!\left[(T-\nu+1)^+\mid\cF_{\nu-1}\right].
\]
Since \(p>1/b\),
\[
q=\frac{1-p}{b-1}<\frac1b,
\]
so the post-change success probability after time one is \(\beta_\mu(q)<\beta_\mu(1/b)\). Conditional on survival to any changepoint \(\nu\geq2\), the remaining delay is geometric with mean \(1/\beta_\mu(q)\). At \(\nu=1\), its mean is
\[
1+\frac{1-\beta_\mu(p)}{\beta_\mu(q)}
\leq \frac1{\beta_\mu(q)}.
\]
Consequently the Lorden worst-case delay of the fast-start chart is
\begin{equation}\label{eq:fast-lorden}
\mathcal D_{\mathrm L}(T)=\frac1{\beta_\mu(q)}
>\frac1{\beta_\mu(1/b)},
\end{equation}
where the right-hand side is the Lorden delay of the stationary strong chart. 

In summary, when comparing the same pair of charts, \emph{the extra freedom provided by the weak e-detector can strictly improve immediate-change power, but it strictly worsens the minimax criterion.}
For example, take \(b=100\), \(p=0.1\), and \(\mu=1\). Then \(a\approx1.282\), \(q\approx0.00909\), and \(c\approx2.362\). The immediate-change power is
\[
\beta_1(0.1)\approx0.389,
\]
whereas the maximum over all strongly representable rules is
\[
\beta_1(0.01)\approx0.0924.
\]
The continuation power is \(\beta_1(q)\approx0.0866\), giving Lorden delay approximately \(11.5\), compared with approximately \(10.8\) for the stationary strong chart. This is a precise startup-versus-minimax tradeoff, closely related to the fast-initial-response and head-start principles studied for CUSUM charts by \citet{lucascrosier1982}. The weak rule can approach perfect first-observation power: as \(p\uparrow1\), the power \(\beta_\mu(p)\uparrow1\), while every scale-\(b\) strong rule remains capped at the fixed value \(\beta_\mu(1/b)\). The price is explicit: \(q\downarrow0\), so the Lorden delay in \eqref{eq:fast-lorden} diverges.
\end{example}


The Gaussian example proves that the weak--strong gap is statistically active for immediate-change power, while also showing that this particular use of the extra freedom is counterproductive under Lorden's minimax criterion. 

The phenomenon is completely general. For any \(b>1\) and any \(p\in(1/b,1)\), choose a sufficiently large integer \(L\) and let \(T=1\) with probability \(p\) and \(T=L\) otherwise. Then \(\E[T]\geq b\), while \(P(T=1)=p>1/b\). The large late value subsidizes an arbitrarily aggressive early false-alarm probability.

\section{Own level-crossing validity does not recover optional validity}\label{sec:eprocess-analogy}

One might wonder whether restricting attention to a process's own level crossings is somehow sufficient to recover an all-stopping-times property. It is not. The distinction already appears for e-processes in a two-point example.

Call a nonnegative adapted process \(E=(E_n)_{n\geq0}\), with \(E_0=1\), an \emph{own-crossing e-certificate}. The normalization \(E_0=1\) is conventional for this comparison; ordinary e-processes, including the canonical indicator process in \eqref{eq:canonical-eprocess-intro}, need only satisfy \(\E[E_0]\leq1\) and may start at zero. We require that, for every \(c>1\),
\begin{equation}\label{eq:own-crossing-e}
\E_P\bigl[E_{\tau_c(E)}^\dagger\bigr]\leq1,
\qquad
\tau_c(E):=\inf\{n\geq1:E_n\geq c\}.
\end{equation}
Every e-process is an own-crossing e-certificate, because the e-process inequality may be evaluated at \(\tau_c(E)\). The converse fails.

\begin{example}[Own crossings miss a fixed-time violation]\label{ex:own-crossing-e}
Let \(Y\sim\mathrm{Bernoulli}(1/2)\), reveal \(Y\) at time one, set \(E_0=1\), and for every \(n\geq1\) let
\[
E_n:=
\begin{cases}
1.3,&Y=1,\\
0.9,&Y=0.
\end{cases}
\]
For \(1<c\leq1.3\), the process crosses \(c\) at time one only on \(\{Y=1\}\), so
\[
\E[E_{\tau_c(E)}^\dagger]=1.3\cdot\frac12=0.65\leq1.
\]
For \(c>1.3\), it never crosses and the left side is zero. Hence \(E\) is an own-crossing e-certificate. It is not an e-process, because the deterministic stopping time \(\tau\equiv1\) gives
\[
\E[E_1]=\frac12(1.3)+\frac12(0.9)=1.1>1.
\]
\end{example}

The own-crossing condition ignores the branch on which the process remains below the threshold. A fixed-time expectation sees both branches. The same mechanism separates weak from strong e-detectors: weak validity controls precisely the stopped values needed for the process's own ARL guarantees, while strong validity controls every way an observer might choose a horizon after seeing the data.

This example also clarifies the nature of the weak universality theorem. Weak e-detectors are not an alternative characterization of strong e-detectors, just as own-crossing e-certificates are not an alternative characterization of e-processes. They are the exact objects associated with a weaker target guarantee.

\section{The exact universality class of strong e-detectors}\label{sec:strong-universality}

\Cref{prop:strong-optional} gives a necessary condition for representability by a strong e-detector. It is also sufficient.

\begin{definition}[Optional-horizon linear false-alarm control]\label{def:oh}
Fix \(b>0\). A stopping time \(T\) has \emph{optional-horizon linear false-alarm control at scale \(b\)} over \(\cP\) if, for every \(P\in\cP\) and every stopping time \(\sigma\),
\begin{equation}\label{eq:oh-def}
P(T\leq\sigma)\leq \frac{\E_P[\sigma]}{b}.
\end{equation}
\end{definition}

The terminology in \Cref{def:oh} is descriptive rather than standard. The key feature is that the comparison horizon may be data-dependent. Taking \(\sigma\equiv t\) gives the deterministic linear tail inequality
\begin{equation}\label{eq:det-tail}
P(T\leq t)\leq \frac{t}{b},\qquad t\in\N.
\end{equation}
Taking \(\sigma=T\) shows that optional-horizon control includes ordinary ARL control.

\begin{proposition}[Optional-horizon control implies ARL]\label{prop:oh-arl}
If \(T\) satisfies \eqref{eq:oh-def}, then
\[
\inf_{P\in\cP}\E_P[T]\geq b.
\]
\end{proposition}

\begin{proof}
Fix \(P\). If \(\E_P[T]=\infty\), the claim is immediate. Otherwise, \(T\) is an admissible choice of \(\sigma\) and \(P(T\leq T)=1\), so \(1\leq \E_P[T]/b\).
\end{proof}

We can now state the second exact representation theorem.

\begin{theorem}[Universality of strong e-detectors]\label{thm:strong-universality}
Fix \(b>0\) and a stopping time \(T\). The following statements are equivalent.

\begin{enumerate}[label=\textup{(\roman*)},leftmargin=2.4em]
\item \(T\) has optional-horizon linear false-alarm control at scale \(b\), as in \eqref{eq:oh-def}.
\item There exists a strong \(\cP\)-e-detector \(M\) such that \(T=T_b(M)\).
\end{enumerate}

When \textup{(i)} holds, the canonical process
\begin{equation}\label{eq:canonical-strong}
M_n^{T,b}:=b\1\{T\leq n\}
\end{equation}
is a strong \(\cP\)-e-detector and crosses level \(b\) exactly at \(T\).
\end{theorem}

\begin{proof}
The implication \textup{(ii)}\(\Rightarrow\)\textup{(i)} is \Cref{prop:strong-optional}.

For the converse, assume \textup{(i)} and define \(M^{T,b}\) by \eqref{eq:canonical-strong}. As in the proof of \Cref{thm:weak-universality}, the level-\(b\) crossing time is exactly \(T\). Now fix \(P\in\cP\) and an integrable stopping time \(\tau\). Since \(\tau<\infty\) almost surely,
\[
M_\tau^{T,b}=b\1\{T\leq\tau\}.
\]
Therefore, by optional-horizon control,
\[
\E_P[M_\tau^{T,b}]
=bP(T\leq\tau)
\leq \E_P[\tau].
\]
Thus \(M^{T,b}\) is strong.
\end{proof}

The theorem identifies the precise ``tail bound in addition to ARL'' that recovers the original e-detector. Importantly, the deterministic family \eqref{eq:det-tail} is not sufficient; one needs its optional-horizon strengthening.

\subsection{Optimal stopping and the maximal strong scale}\label{subsec:optimal-stopping}

For a fixed null law \(P\), define the optimal-stopping value
\begin{equation}\label{eq:optimal-stopping-value}
V_b^P(T)
:=
\sup_{\sigma\text{ bounded}}
\E_P\!\left[b\1\{T\leq\sigma\}-\sigma\right].
\end{equation}
The horizon \(\sigma\equiv0\) is admissible and alarm times are positive, so \(V_b^P(T)\geq0\).

\begin{proposition}[Optimal-stopping formulation]\label{prop:optimal-stopping-formulation}
For an alarm time \(T\) and a scale \(b>0\), the following are equivalent for a null law \(P\):
\begin{enumerate}[label=\textup{(\roman*)},leftmargin=2.4em]
\item \(T\) has optional-horizon linear false-alarm control at scale \(b\) under \(P\);
\item \(V_b^P(T)=0\);
\item \(b\leq b_P^*(T)\).
\end{enumerate}
Thus ``stop immediately'' solves the optimal-stopping problem with reward \(b\1\{T\leq n\}-n\) exactly when \(T\) is strongly representable at scale \(b\).
\end{proposition}

\begin{proof}
Optional-horizon control is equivalent, by \Cref{lem:localization}, to
\[
\E_P[b\1\{T\leq\sigma\}-\sigma]\leq0
\]
for every bounded stopping time. Since the zero horizon attains zero, this is equivalent to \(V_b^P(T)=0\). Equivalence with \textup{(iii)} is \Cref{prop:strong-scale}, and equivalence with strong representation is \Cref{thm:strong-universality}.
\end{proof}

The quantity \(b_P^*(T)\) is therefore a ratio-form optimal-stopping value. Unlike ARL, it is filtration-sensitive: an observer may lower the ratio by spending monitoring time selectively on histories on which the alarm is likely. For Markovian charts, the numerator and denominator can be handled by standard fractional optimal-stopping or dynamic-programming methods.

\subsection{Which run-length laws are strongly representable?}\label{subsec:run-length-laws}

Optional-horizon validity is a property of a stopping time together with its filtration, not merely of its marginal law. Nevertheless, the law admits an exact characterization when no information is available beyond whether the alarm has already occurred. Define the \emph{minimal alarm filtration}
\begin{equation}\label{eq:minimal-alarm-filtration}
\cG_n^T:=\sigma\bigl(\{T\leq j\}:1\leq j\leq n\bigr),
\qquad n\in\Nzero.
\end{equation}
Before the alarm, the surviving set \(\{T\geq n\}\) is a single atom of \(\cG_{n-1}^T\).

\begin{lemma}[Horizons in the minimal alarm filtration]\label{lem:minimal-horizons}
If \(\sigma\) is a stopping time for \((\cG_n^T)\), then there is a deterministic \(m\in\Nzero\cup\{\infty\}\) such that
\begin{equation}\label{eq:minimal-horizon-form}
\sigma\wedge T=T\wedge m.
\end{equation}
\end{lemma}

\begin{proof}
For every \(n\geq1\),
\[
\{\sigma\wedge T\geq n\}
=\{\sigma\geq n\}\cap\{T\geq n\}
\in\cG_{n-1}^T.
\]
Because \(\{T\geq n\}\) is an atom of \(\cG_{n-1}^T\), this event is either empty or all of \(\{T\geq n\}\). Define
\[
m:=\sup\left\{n\geq0:
\{\sigma\wedge T\geq j\}=\{T\geq j\}
\text{ for every }j\leq n\right\}.
\]
If the first failure occurs at \(m+1\), then \(\{\sigma\wedge T\geq m+1\}=\varnothing\) while \(\{T\geq m+1\}\neq\varnothing\); because the left-hand events decrease, they remain empty thereafter, even if \(\{T\geq n\}\) later also becomes empty. Hence
\[
\{\sigma\wedge T\geq n\}
=
\begin{cases}
\{T\geq n\},&n\leq m,\\
\varnothing,&n>m,
\end{cases}
\]
which is exactly the tail-event description of \(T\wedge m\).
\end{proof}

\begin{theorem}[Truncated-mean characterization of strong run-length laws]\label{thm:run-length-law}
Fix \(b>0\) and an alarm time \(T\).
\begin{enumerate}[label=\textup{(\alph*)},leftmargin=2.2em]
\item Strong representability at scale \(b\) implies, for every \(P\in\cP\) and every \(m\in\N\),
\begin{equation}\label{eq:truncated-mean-law}
bP(T\leq m)\leq \E_P[T\wedge m].
\end{equation}
\item If the working filtration is the minimal alarm filtration \((\cG_n^T)\), then \eqref{eq:truncated-mean-law} for every \(P\in\cP\) and \(m\in\N\) is equivalent to optional-horizon linear false-alarm control at scale \(b\), and hence to strong e-detector representability at level \(b\). The canonical representative is \(M_n=b\1\{T\leq n\}\).
\end{enumerate}
Consequently, a probability law on \(\N\cup\{\infty\}\) is the null run-length law of some strong e-detector at scale \(b\), on its own natural alarm filtration, if and only if it satisfies \eqref{eq:truncated-mean-law}.
\end{theorem}

\begin{proof}
For \textup{(a)}, apply optional-horizon control to the bounded stopping time \(\sigma=T\wedge m\). Since \(\{T\leq T\wedge m\}=\{T\leq m\}\), this gives \eqref{eq:truncated-mean-law}.

For \textup{(b)}, necessity follows from part \textup{(a)}. Conversely, let \(\sigma\) be any integrable \((\cG_n^T)\)-stopping time. Replacing \(\sigma\) by \(\sigma\wedge T\) preserves the event \(\{T\leq\sigma\}\) and can only reduce its expectation. By \Cref{lem:minimal-horizons}, \(\sigma\wedge T=T\wedge m\) for some deterministic \(m\in\Nzero\cup\{\infty\}\). If \(m=0\), then \(P(T\leq\sigma)=0\) because alarm times are positive. If \(1\leq m<\infty\), then
\[
bP(T\leq\sigma)
=bP(T\leq m)
\leq\E_P[T\wedge m]
=\E_P[\sigma\wedge T]
\leq\E_P[\sigma].
\]
If \(m=\infty\), then \(\sigma\wedge T=T\). Letting the finite truncation level tend to infinity in \eqref{eq:truncated-mean-law} gives
\[
bP(T<\infty)\leq\E_P[T]=\E_P[\sigma\wedge T]\leq\E_P[\sigma]
\]
by monotone convergence. Thus \(T\) has optional-horizon control, and \Cref{thm:strong-universality} supplies the canonical strong representative.
\end{proof}

\begin{corollary}[Strong scale in the minimal alarm filtration]\label{cor:minimal-strong-scale}
Under the minimal alarm filtration and a fixed null law \(P\),
\begin{equation}\label{eq:minimal-strong-scale}
b_P^*(T)
=
\inf_{\substack{m\in\N\cup\{\infty\}\\P(T\leq m)>0}}
\frac{\E_P[T\wedge m]}{P(T\leq m)}.
\end{equation}
Equivalently,
\begin{equation}\label{eq:minimal-value-function}
V_b^P(T)
=
\sup_{m\in\Nzero}
\left\{bP(T\leq m)-\E_P[T\wedge m]\right\}.
\end{equation}
\end{corollary}

\begin{proof}
Replacing any horizon \(\sigma\) by \(\sigma\wedge T\) preserves \(\{T\leq\sigma\}\) and weakly decreases its expectation. By \Cref{lem:minimal-horizons}, the replacement equals \(T\wedge m\) for a deterministic \(m\). Substitution into \eqref{eq:strong-scale} and \eqref{eq:optimal-stopping-value} proves both formulas; the case \(m=\infty\) follows by monotone convergence.
\end{proof}

The condition has a useful budget interpretation. For one null law, write
\begin{equation}\label{eq:run-length-reserve}
D_m
:=\E[T\wedge m]-bP(T\leq m)
=\sum_{n=1}^m\bigl(P(T\geq n)-bP(T=n)\bigr).
\end{equation}
If \(h_n=P(T=n\mid T\geq n)\), then
\[
D_m=\sum_{n=1}^m P(T\geq n)(1-bh_n).
\]
Strong representability under the minimal filtration is exactly the requirement that the cumulative reserve \(D_m\) never become negative. A hazard may exceed \(1/b\) at some time, provided earlier low hazards have accumulated enough reserve. For \(b\geq1\), the geometric law with hazard \(1/b\) has \(D_m=0\) for every \(m\). Conversely, equality at every horizon forces, after differencing successive equalities,
\[
P(T=n)=\frac1b P(T\geq n),
\]
so the geometric law is unique. By contrast, for integer \(b\), the deterministic law \(T\equiv b\) first accumulates reserve and then spends it in one step. 


There is also a classical reliability-theoretic interpretation. Suppose the run length is exactly calibrated, with \(\mu:=\E[T]=b<\infty\), and define its discrete equilibrium lifetime \(T^{\mathrm{eq}}\) by
\[
P(T^{\mathrm{eq}}=n):=\frac{P(T\geq n)}{\mu},
\qquad n\geq1.
\]
Then
\[
P(T^{\mathrm{eq}}\leq m)
=\frac{\E[T\wedge m]}{\mu},
\]
so \eqref{eq:truncated-mean-law} is equivalent to \(T\geq_{\mathrm{st}}T^{\mathrm{eq}}\). This is precisely the discrete new-better-than-used-in-expectation (NBUE) condition; see \citet{marshallolkin2007,nairsankaranpreeth2012}. Hence, among null run lengths with exact mean \(b\), strong representability on the natural alarm filtration is equivalent to NBUE. The geometric law is the boundary case, while deterministic run lengths are strict NBUE examples.

The truncated inequalities immediately imply both familiar consequences:
\[
bP(T\leq m)\leq\E[T\wedge m]\leq m
\]
gives the linear deterministic tail bound, while letting \(m\to\infty\) gives the ARL bound. At \(m=1\), the inequality already rules out the early-alarm lottery in \Cref{ex:early-lottery}. More strongly, the family is strictly stronger than the deterministic tail and ARL requirements taken together: in \Cref{ex:det-not-optional}, \(b=10\), \(P(T\leq5)=1/2\), and \(\E[T\wedge5]=4\), so \eqref{eq:truncated-mean-law} fails even though every deterministic tail bound and the ARL constraint hold.

\begin{corollary}[Rigidity at a saturated deterministic horizon]\label{cor:tail-rigidity}
If \(T\) is strongly representable at scale \(b\) and
\[
P(T\leq m)=\frac{m}{b}
\]
for some \(m\in\N\), then \(T\geq m\) almost surely and
\[
P(T=m)=\frac{m}{b}.
\]
Thus the linear tail envelope can be saturated at a horizon only by a rule that never alarms before that horizon.
\end{corollary}

\begin{proof}
Strong representability gives
\[
m=bP(T\leq m)\leq\E[T\wedge m]\leq m.
\]
Equality throughout forces \(T\wedge m=m\) almost surely, hence \(T\geq m\). Therefore \(\{T\leq m\}=\{T=m\}\).
\end{proof}

\begin{example}[A deterministic strong run length]\label{ex:deterministic-strong}
Let \(b\in\N\), set \(T\equiv b\), and define \(M_n=b\1\{n\geq b\}\). For every integrable stopping time \(\tau\),
\[
\E[M_\tau]
=bP(\tau\geq b)
\leq\E[\tau],
\]
because \(\tau\geq b\) on \(\{\tau\geq b\}\). Thus \(M\) is strong and its level-\(b\) crossing time is the deterministic value \(b\). In particular, optional-horizon validity alone does not force a geometric or approximately exponential null run length.
\end{example}

The filtration qualification in \Cref{thm:run-length-law} is essential. Let \(b=10\) and let \(T\) equal \(5\) or \(15\), each with probability one half. Its law satisfies \eqref{eq:truncated-mean-law}: for \(5\leq m<15\), both sides are bounded by
\[
10P(T\leq m)=5
\leq\frac{5+m}{2}=\E[T\wedge m],
\]
and equality holds once \(m\geq15\). Hence the law is strongly representable under its minimal alarm filtration. If, however, the branch indicator is revealed at time one, the horizon \(\sigma=5\) on \(\{T=5\}\) and \(\sigma=1\) on \(\{T=15\}\) is a stopping time with
\[
P(T\leq\sigma)=\frac12
>\frac{3}{10}
=\frac{\E[\sigma]}{10}.
\]
The same marginal law therefore fails strong validity in the enriched filtration. Optional-horizon control is sensitive to anticipatory side information in exactly the way its definition suggests.

\subsection{Snell envelopes and martingale domination}

Strong validity has a checkable optimal-stopping characterization. The following clocked version is the e-detector analogue of martingale-domination results for anytime-valid inference, and its proof uses the Snell-envelope mechanism emphasized by \citet{ramdasruflarssonkoolen2022}.

\begin{theorem}[Snell and martingale domination]\label{thm:snell-strong}
Fix \(P\in\cP\). Let \(C\) be a clock with \(C_0=0\) and \(\E_P[C_n]<\infty\) for every finite \(n\), and let \(M\) be nonnegative and adapted with \(M_0=0\). The following are equivalent.
\begin{enumerate}[label=\textup{(\roman*)},leftmargin=2.4em]
\item \(M\) is a strong \(C\)-clocked \(P\)-e-process.
\item There exists an integrable \(P\)-supermartingale \(Y=(Y_n)_{n\geq0}\) such that
\begin{equation}\label{eq:snell-majorant}
Y_0=0,
\qquad
M_n-C_n\leq Y_n
\quad P\text{-almost surely for every }n.
\end{equation}
\item There exists an integrable \(P\)-martingale \(L=(L_n)_{n\geq0}\) such that
\begin{equation}\label{eq:martingale-majorant}
L_0=0,
\qquad
M_n\leq C_n+L_n
\quad P\text{-almost surely for every }n.
\end{equation}
\end{enumerate}
For the calendar clock \(C_n=n\), this characterizes strong e-detectors.
\end{theorem}

For a composite null class, the theorem applies pointwise and yields a possibly different majorant \(Y^P\) or \(L^P\) for every \(P\in\cP\). It does not supply one process that is simultaneously a supermartingale or martingale under all null laws. Without additional structure, such a common majorant cannot generally be expected; this parallels the distinction between pointwise and common supermartingale domination for composite-null e-processes \citep{rufetal2023}.

\begin{proof}
Assume \textup{(i)} and write \(X_n=M_n-C_n\). Strong clocked validity at deterministic time \(j\) gives \(\E_P[M_j]\leq\E_P[C_j]<\infty\). Consequently, for every finite \(N\),
\[
\E_P\!\left[\max_{0\leq j\leq N}|X_j|\right]
\leq \sum_{j=0}^N\E_P[M_j+C_j]
<\infty.
\]
Thus the finite-horizon Snell envelopes below are well defined and admit optimal stopping times. For integers \(N\geq n\), let
\[
Y_n^{(N)}
:=\esssup_{\tau:\,n\leq\tau\leq N}
\E_P[X_\tau\mid\cF_n].
\]
This process dominates \(X_n\) and is a supermartingale over \(n\leq N\).

Let \(\tau_{n,N}^*\) be an optimal finite-horizon stopping time and put \(A=\{Y_n^{(N)}>0\}\in\cF_n\). Stop at \(\tau_{n,N}^*\) on \(A\) and at \(n\) on \(A^c\). Strong clocked validity gives
\[
0\geq
\E_P\bigl[\1_A Y_n^{(N)}+\1_{A^c}X_n\bigr],
\]
so
\[
\E_P[(Y_n^{(N)})^+]
\leq\E_P[X_n^-]
\leq\E_P[C_n].
\]
Also \(Y_n^{(N)}\geq X_n\geq-C_n\). Since \(Y_n^{(N)}\) is nondecreasing in \(N\), the limit
\[
Y_n:=\lim_{N\to\infty}Y_n^{(N)}
\]
is finite and integrable. Conditional Fatou's lemma, using the integrable lower bound \(-C_{n+1}\), shows that \(Y_n\geq\E_P[Y_{n+1}\mid\cF_n]\). Thus \(Y\) is a supermartingale and dominates \(X\). At time zero, \(X_0=0\). The same pasting argument with fallback time zero gives \(Y_0^{(N)}\leq0\), while immediate stopping gives \(Y_0^{(N)}\geq0\); hence \(Y_0=0\). This proves \textup{(ii)}.

By the discrete-time Doob decomposition, an integrable supermartingale with \(Y_0=0\) can be written as \(Y=L-A\), where \(L\) is an integrable martingale with \(L_0=0\) and \(A\) is predictable and nondecreasing with \(A_0=0\). Hence \(Y\leq L\), and \textup{(ii)} implies \textup{(iii)}.

Finally, suppose \textup{(iii)} holds. For every bounded stopping time \(\tau\), optional sampling gives
\[
\E_P[M_\tau]
\leq\E_P[C_\tau]+\E_P[L_\tau]
=\E_P[C_\tau].
\]
Clocked localization extends the inequality to arbitrary stopping times, proving \textup{(i)}.
\end{proof}

\begin{corollary}[General-filtration certificate for a run length]\label{cor:run-length-martingale-certificate}
An alarm time \(T\) is strongly representable at scale \(b\) under \(P\) if and only if there exists an integrable martingale \(L\), with \(L_0=0\), such that
\[
b\1\{T\leq n\}\leq n+L_n
\qquad P\text{-almost surely for every }n.
\]
Equivalently, the reward process \(b\1\{T\leq n\}-n\) has a supermartingale majorant starting from zero. This is the Snell certificate for the optimal-stopping formulation in \Cref{prop:optimal-stopping-formulation}.
\end{corollary}

\begin{remark}[A sufficient drift condition, but not a characterization]\label{rem:drift-certificate}
If \(M_n-n\) is itself a supermartingale, then \Cref{thm:snell-strong} applies with \(Y_n=M_n-n\). For the canonical detector \(M_n=b\1\{T\leq n\}\), the one-step drift condition
\[
P(T=n\mid\cF_{n-1})\leq b^{-1}\1\{T\geq n\}
\]
implies exactly this supermartingale property. The converse fails, so the Snell envelope is doing genuine work. For example, let
\[
M_0=0,\qquad M_1=0,\qquad M_2=2,\qquad M_n=0\quad(n\geq3).
\]
Then \(M_\tau=2\1\{\tau=2\}\), so \(\E[M_\tau]\leq\E[\tau]\) for every stopping time and \(M\) is strong. Yet
\[
\E[M_2-2\mid\cF_1]=0>-1=M_1-1,
\]
so \(M_n-n\) is not a supermartingale.
\end{remark}

The domination theorem makes convexity transparent: convex combinations of martingale or supermartingale majorants remain valid majorants. No analogous all-horizon certificate exists for the weak class, consistent with \Cref{prop:convexity-separation,prop:running-maxima}.

\subsection{Deterministic-horizon bounds are not enough}

\begin{example}[A deterministic tail bound defeated by an adaptive horizon]\label{ex:det-not-optional}
Let \(\alpha=0.1\), so \(b=10\). Let \(V\sim\mathrm{Bernoulli}(1/2)\), let \(U\) be independent and uniform on \(\{1,2,3,4,5\}\), and reveal \((U,V)\) at time one. Define
\[
T:=
\begin{cases}
U,&V=1,\\
20,&V=0.
\end{cases}
\]
Then for \(1\leq t\leq5\),
\[
P(T\leq t)=\frac12\frac{t}{5}=0.1t,
\]
while for \(5<t<20\), \(P(T\leq t)=1/2\leq0.1t\), and for \(t\geq20\), \(P(T\leq t)=1\leq0.1t\). Thus the deterministic bound \(P(T\leq t)\leq t/10\) holds at every time. Also,
\[
\E[T]=\frac12\cdot3+\frac12\cdot20=11.5\geq10.
\]

Now choose the data-dependent stopping horizon
\[
\sigma:=
\begin{cases}
5,&V=1,\\
1,&V=0.
\end{cases}
\]
This is a stopping time because \(V\) is observed at time one. On \(\{V=1\}\), one always has \(T=U\leq5=\sigma\); on \(\{V=0\}\), one has \(T=20>1=\sigma\). Consequently,
\[
P(T\leq\sigma)=\frac12,
\qquad
\frac{\E[\sigma]}{10}
=\frac{\frac12\cdot5+\frac12\cdot1}{10}
=0.3.
\]
The optional-horizon inequality fails. By \Cref{thm:weak-universality}, \(T\) is representable by a weak e-detector; by \Cref{thm:strong-universality}, it is not representable by a strong one.
\end{example}

The adaptive horizon spends more observation time on the branch where an alarm is likely and stops quickly on the other branch. No deterministic horizon can express this selection effect. This is exactly the distinction between a collection of marginal tail inequalities and an optional-stopping guarantee.

\subsection{A sufficient conditional-hazard condition}

Optional-horizon control may initially look difficult to verify. A simple conditional hazard condition is sufficient and gives a familiar example.

\begin{proposition}[Bounded conditional false-alarm hazard]\label{prop:hazard}
Let \(b>0\), and suppose that for every \(P\in\cP\) and every \(n\geq1\),
\begin{equation}\label{eq:hazard}
P(T=n\mid\cF_{n-1})
\leq
\frac{1}{b}\1\{T\geq n\}
\qquad P\text{-almost surely}.
\end{equation}\footnote{For an upper bound, the at-risk indicator is algebraically redundant because the stopping-time property already forces \(P(T=n\mid\cF_{n-1})=0\) on \(\{T<n\}\). We retain it to display the hazard interpretation. In the lower bound of \eqref{eq:two-sided-hazard}, the indicator is essential.}
Then \(T\) has optional-horizon linear false-alarm control at scale \(b\).
\end{proposition}

\begin{proof}
Fix \(P\) and an integrable stopping time \(\sigma\). The event \(\{\sigma\geq n\}\) belongs to \(\cF_{n-1}\). By Tonelli's theorem and \eqref{eq:hazard},
\begin{align*}
P(T\leq\sigma)
&=\sum_{n=1}^{\infty}\E_P\bigl[\1\{\sigma\geq n\}\1\{T=n\}\bigr]\\
&=\sum_{n=1}^{\infty}\E_P\!\left[\1\{\sigma\geq n\}P(T=n\mid\cF_{n-1})\right]\\
&\leq \frac1b\sum_{n=1}^{\infty}P(\sigma\geq n,\,T\geq n)\\
&\leq \frac1b\sum_{n=1}^{\infty}P(\sigma\geq n)
=\frac{\E_P[\sigma]}{b}.
\end{align*}
\end{proof}

\begin{example}[Geometric false alarms]\label{ex:geometric}
Let \(b\geq1\). Suppose that under every null law the procedure alarms at each time with conditional probability \(1/b\), given survival to that time. Then \(T\) is geometric with mean \(b\), \eqref{eq:hazard} holds with equality, and the canonical indicator process \(b\1\{T\leq n\}\) is a strong e-detector. Notice that \(P(T<\infty)=1\): strong e-detector validity is compatible with eventual false alarm almost surely, unlike e-process/PFA validity. What it controls is the rate at which false-alarm probability can be accumulated, including after optional selection of the horizon.
\end{example}

\begin{proposition}[Nonasymptotic geometric sandwich]\label{prop:geometric-sandwich}
Suppose constants \(0\leq\underline q\leq\overline q\leq1\) satisfy, for every \(n\geq1\),
\begin{equation}\label{eq:two-sided-hazard}
\underline q\,\1\{T\geq n\}
\leq P(T=n\mid\cF_{n-1})
\leq\overline q\,\1\{T\geq n\}
\quad P\text{-almost surely}.
\end{equation}
Then
\begin{equation}\label{eq:geometric-sandwich}
(1-\overline q)^n
\leq P(T>n)
\leq(1-\underline q)^n,
\qquad n\geq0.
\end{equation}
In particular, an upper hazard bound \(1/b\), with \(b\geq1\), implies that \(T\) stochastically dominates a geometric random variable of mean \(b\). If additionally \(\E[T]=b\), then \(P(T>n)=(1-1/b)^n\) for every \(n\), and the null run length is exactly geometric.
\end{proposition}

\begin{proof}
Since \(\{T\geq n\}\in\cF_{n-1}\), subtracting the conditional probability of \(\{T=n\}\) from \(\1\{T\geq n\}\) gives
\[
(1-\overline q)\1\{T\geq n\}
\leq
\E[\1\{T>n\}\mid\cF_{n-1}]
\leq
(1-\underline q)\1\{T\geq n\}.
\]
Take expectations and iterate. For the last claim, the one-sided bound gives \(P(T>n)\geq(1-1/b)^n\). Summing over \(n\geq0\) gives \(\E[T]\geq b\). Equality of the sums forces equality term by term.
\end{proof}

The geometric sandwich is stronger than the linear deterministic bound \(P(T\leq t)\leq t/b\). It gives a finite-sample route to the approximately exponential run-length behavior studied asymptotically by \citet{pollaktartakovsky2009exit}. The optional-horizon inequality alone does not imply such exponentiality: \Cref{ex:deterministic-strong} is a strong e-detector whose run length is a point mass.

\section{Discussion}\label{sec:discussion}

\subsection{On weak versus strong e-detectors}

The clock theorems give two exact correspondences. Expected-clock lower bounds are represented by weak own-crossing objects; optional-clock false-alarm bounds are represented by strong all-stopping-times objects. For calendar time, these become ARL control and optional-horizon linear false-alarm control, respectively.

The canonical weak detector \(b\1\{T\leq n\}\) proves formal completeness, but its validity condition is close to the ARL statement it certifies. It is zero before the alarm, supplies no graded evidence, and does not simplify verification of a pre-existing rule. More importantly, \Cref{prop:convexity-separation} shows that weak e-detectors are not stable under averaging. Weak detectors can always be put in nondecreasing form by taking running maxima, but \Cref{prop:running-maxima} shows that this operation need not preserve strong validity. These asymmetric closure properties reinforce the same conclusion: the weak class is a representation language, not a robust algebra for constructing procedures.

Strong e-detectors are different. Their all-horizon inequality is convex, survives independent enlargement by \Cref{lem:independent-enlargement}, is preserved by the aggregation operations used in \citet{shin2024}, and has the optimal-stopping and martingale certificates in \Cref{prop:optimal-stopping-formulation} and \Cref{thm:snell-strong}. The canonical strong indicator remains degenerate, but the class contains nondegenerate evidence-accumulating procedures with favorable detection delay. 

\subsection{The operational content of strong validity}

Every strong e-detector threshold obeys
\(
P(T_b\leq t)\leq t/b
\)
at fixed horizons and the  optional-horizon inequality at data-dependent horizons. Hence any e-detector-based CUSUM, Shiryaev--Roberts, or mixture procedure automatically avoids the extreme front-loading permitted by an ARL constraint. This matters in deployments with a finite mission, batch, or monitoring window, and it addresses the concern that a large ARL can conceal high early false-alarm risk \citep{mei2008}.

The maximal strong scale makes the protection quantitative. The largest scale at which a given rule can be advertised as strongly valid is
\[
b_P^*(T)=\inf_\sigma \frac{\E_P[\sigma]}{P(T\leq\sigma)},
\]
whereas the largest weak scale is its ARL. Thus \(\E_P[T]/b_P^*(T)\) measures how much of the nominal ARL is created by histories that an adaptive horizon can avoid. For the fast-start chart, \(b_P^*(T)=1/p\): at \(b=100\) and \(p=0.1\), an ARL of \(100\) supports only strong calibration at scale \(10\).

The self-censoring horizons \(T_b\wedge m\) already give the sharper diagnostic
\[
bP(T_b\leq m)\leq\E[T_b\wedge m].
\]
Under the minimal alarm filtration, \Cref{thm:run-length-law} shows that these inequalities are not merely consequences but a complete characterization of the attainable run-length laws. The reserve process in \eqref{eq:run-length-reserve} distinguishes a geometric rule, which spends its budget locally at every step, from a deterministic rule, which can accumulate budget and spend it later. \Cref{cor:tail-rigidity} adds a sharp no-front-loading statement: if the linear tail envelope is saturated at one horizon, then the procedure cannot alarm before that horizon.

Optional-horizon control may be read as an ARL analogue of anytime validity. An observer may allocate more monitoring time to sample paths that appear likely to alarm and less time elsewhere, as in \Cref{ex:det-not-optional}. A strong e-detector remains valid under every such allocation. The clock view makes the distinction concise: an e-process has a fixed unit budget; an e-detector receives one unit of budget per expected unit of monitored time.

The exponential-clock construction in \Cref{prop:exp-clock-pinning} supplies a complementary practical point. Convexity and stability under independent enlargement permit a nondegenerate data-driven strong detector to be mixed with an external random clock, producing an actual ARL in \([\gamma,(1+\eta)\gamma+1]\) rather than only a lower bound. 
The price is explicit: the auxiliary component forces an eventual null crossing and places a Poisson-tail upper bound on the run length, so the mixture should be understood as calibrated regularization rather than free additional evidence.

\subsection{Does the weak--strong gap buy detection power?}

\Cref{ex:fast-start-shewhart} gives an exact rather than benchmark-dependent separation. At the same ARL, its first-observation power exceeds the maximum possible for every strongly representable rule in the Gaussian shift experiment. Indeed, as \(p\uparrow1\), this power approaches one while the strong-class optimum remains fixed at \(\beta_\mu(1/b)\). Thus the weak rule can approach the maximal possible power one while the strong-class optimum stays fixed strictly below one. Its price is equally stark: the continuation hazard tends to zero and the Lorden delay diverges. Thus the weak freedom is provably active for an immediate change and provably harmful for this particular minimax comparison.

Existing strong e-detector constructions can attain first-order asymptotically optimal delay in broad nonparametric and composite settings \citep{shin2024,ramramdas2026}, so strong validity need not impose a first-order asymptotic cost in those models.


The weak representation theorem is scale-specific. For each \(b\), it produces a process representing one ARL-valid \(T\) at level \(b\). A whole family \(\{T_b\}_{b>0}\) need not be realizable as all thresholds of one weak detector, even when \(\E[T_b]\geq b\) for every \(b\). Appendix~\ref{app:simultaneous} gives an exact characterization: pathwise threshold coherence is necessary, and the canonical running level must satisfy an additional overshoot inequality. This simultaneous-family problem is a more stringent and genuinely nontrivial notion of universality.

The appendices also show that conjunctions of validity criteria need not require a new canonical statistic. ARL and a deterministic tail envelope are represented by the same indicator process, provided it satisfies both an own-crossing budget and a fixed-time running-maximum budget; see Appendix~\ref{app:conjunction}. Restart-indexed criteria reveal another useful boundary: weak residual ARL has its own field representation, whereas requiring strong validity after every observed history is exactly equivalent to conditional hazard control; see Appendix~\ref{app:residual-arl}.

\section{Conclusion}\label{sec:conclusion}

The appropriate universal object for change detection depends on the false-alarm criterion. A general budget clock makes this dependence explicit. Optional-clock false-alarm control is equivalent to thresholding a strong clocked e-process, while an expected-clock lower bound is equivalent to thresholding a weak own-crossing object.

For the calendar-time clock, the original e-detectors of \citet{shin2024} exactly represent the alarm times satisfying optional-horizon linear false-alarm control. Their thresholds therefore obey nontrivial lower-tail guarantees at every deterministic and data-dependent horizon; they cannot hide aggressive startup false alarms behind a long right tail. The maximal strong scale \(b^*(T)\) quantifies the largest defensible strong calibration of any rule. Under the minimal filtration generated by the alarm, the sharp truncated-mean inequalities \(bP(T\leq m)\leq\E[T\wedge m]\) completely characterize the possible null run-length laws, equivalently through an NBUE condition at exact calibration. The geometric law balances this inequality at every horizon, but deterministic and other nongeometric laws are also possible.

Weak e-detectors exactly represent the larger class of ARL-controlled alarm times. This extra freedom is statistically real: the fast-start Gaussian chart can drive immediate-change power arbitrarily close to one while retaining a fixed ARL, beyond the sharp strong-class cap, but its maximal strong scale collapses and its Lorden delay diverges. Thus ARL alone permits an explicit startup-versus-minimax tradeoff that optional-horizon control rules out.

The notions also differ structurally. Strong e-detectors are convex, stable under independent randomization, and admit Snell-envelope, supermartingale, and martingale-domination certificates. Their convexity supports nondegenerate exponential-clock calibration with explicit two-sided ARL bounds. Weak e-detectors are not convex, although they admit monotone normalization; conversely, running maxima can destroy strong validity. Weak universality should therefore be understood as formal completeness rather than as a general construction method. Bare ARL is sufficiently weak to admit a universal certificate but too weak to force a useful algebra of certificates. The stronger optional-horizon criterion is what supports operational protection, interpretable run-length restrictions, and constructive e-detector methodology.

\subsection*{Acknowledgments} 
The author thanks Ashwin Ram for useful technical conversations and feedback on an early preprint. AI tools were used for developing some of the extensions in the Appendix, some exposition, and creating some illuminating examples, but the author takes responsibility for their correctness.

\clearpage
\appendix

\section*{Guide to the appendices}

The main text develops the common clock abstraction and then focuses on its two calendar-time specializations: weak e-detectors for ARL and strong e-detectors for optional-horizon linear false-alarm control. The appendices serve a different purpose. They test how far the same indicator-certificate method extends to other natural type-I error metrics, and they separate exact representation statements from practically checkable construction principles. None of the main weak--strong results depends on these extensions.

Appendix~\ref{app:weighted-loss} shows that weighted false-alarm losses, including discounted and Bayesian criteria, require no new process class: ordinary e-processes with curved boundaries are universal. Appendix~\ref{app:tail-envelope} treats deterministic run-length tail envelopes using marginal e-values and exhibits their failure under optional stopping. Appendix~\ref{app:e-hazard} develops conditional one-step hazard certificates and connects them back to strong clocked validity. Appendix~\ref{app:local-window} considers rolling-window error control; its survival-conditioned representation is exact but self-referential, so it should be read as a completeness theorem rather than a row-by-row construction recipe. Appendix~\ref{app:residual-arl} treats both weak residual-ARL fields and their strong, historywise counterpart; the latter turns out to be equivalent to a conditional hazard bound. Appendix~\ref{app:counting} handles repeated alarms.

Two new appendices address questions raised by scale-specific universality. Appendix~\ref{app:simultaneous} characterizes when an entire family \(\{T_b\}\) can be represented by one weak detector and shows that scale-wise ARL bounds are insufficient. Appendix~\ref{app:conjunction} represents simultaneous ARL and deterministic-tail requirements. Finally, Appendix~\ref{app:synthesis} summarizes the resulting taxonomy. Terms such as \emph{tail e-envelope}, \emph{e-hazard sequence}, and \emph{rolling e-field} are proposed bookkeeping terminology, not claims of established usage.

\section{Weighted false-alarm losses and curved e-boundaries}\label{app:weighted-loss}

The main text contrasts two familiar criteria: PFA, which bounds the probability of ever alarming, and ARL, which lower-bounds the mean alarm time. A broad intermediate family assigns a deterministic loss to the time at which a false alarm occurs. These criteria require no new process class: ordinary e-processes remain universal, but the appropriate crossing boundary is generally time-dependent.

Let
\[
w:\N\cup\{\infty\}\longrightarrow[0,\infty)
\]
be deterministic, with \(0<w(n)<\infty\) for every finite \(n\), and set \(w(\infty)=0\). The quantity \(w(T)\) may be interpreted as the loss incurred by a false alarm at time \(T\). No monotonicity assumption is needed for the representation theorem, although in applications \(w\) will often be nonincreasing so that early alarms are penalized more heavily.

\begin{theorem}[Universality for weighted stopping losses]\label{thm:weighted-universality}
Fix \(\alpha>0\) and a stopping time \(T\). The following are equivalent.
\begin{enumerate}[label=\textup{(\roman*)},leftmargin=2.4em]
\item
\begin{equation}\label{eq:weighted-loss-control}
\sup_{P\in\cP}\E_P[w(T)]\leq\alpha.
\end{equation}
\item There exists a \(\cP\)-e-process \(E=(E_n)_{n\geq0}\) such that
\begin{equation}\label{eq:weighted-boundary-crossing}
T=\inf\left\{n\geq1:E_n\geq\frac{w(n)}{\alpha}\right\}.
\end{equation}
\end{enumerate}
When \textup{(i)} holds, one may take
\begin{equation}\label{eq:canonical-weighted-eprocess}
E_n^{T,w}:=\frac{w(T)}{\alpha}\1\{T\leq n\},\qquad n\geq0.
\end{equation}
\end{theorem}

\begin{proof}
Assume \textup{(ii)} and let \(T_N:=T\wedge N\). On \(\{T\leq N\}\), the process has crossed its boundary at time \(T\), so
\[
E_{T_N}=E_T\geq \frac{w(T)}{\alpha}.
\]
Since \(E\) is an e-process,
\[
\E_P\bigl[w(T)\1\{T\leq N\}\bigr]
\leq \alpha\E_P[E_{T_N}]
\leq\alpha
\]
for every \(P\in\cP\). Monotone convergence and \(w(\infty)=0\) give \eqref{eq:weighted-loss-control}.

Conversely, assume \textup{(i)} and define \(E^{T,w}\) by \eqref{eq:canonical-weighted-eprocess}. For every stopping time \(\tau\),
\[
\E_P[E_\tau^{T,w}]
=\frac1\alpha\E_P\bigl[w(T)\1\{T\leq\tau\}\bigr]
\leq\frac1\alpha\E_P[w(T)]
\leq1.
\]
Thus \(E^{T,w}\) is an e-process. Before \(T\) it is zero, while at \(T\) it equals \(w(T)/\alpha\), so its first crossing of the curved boundary in \eqref{eq:weighted-boundary-crossing} is exactly \(T\).
\end{proof}

\begin{remark}[PFA is the constant-loss case]\label{rem:pfa-weighted}
Taking \(w(n)\equiv1\) at finite times recovers
\[
\sup_{P\in\cP}P(T<\infty)\leq\alpha
\]
and the usual constant e-process boundary \(1/\alpha\). Thus the familiar PFA universality theorem is the constant-loss member of Theorem~\ref{thm:weighted-universality}.
\end{remark}

\begin{example}[Bayesian false-alarm probability]\label{ex:bayes-weighted}
Suppose a change time \(\nu\) is independent of the pre-change observations and has survival function
\[
\overline\Pi(n):=P(\nu>n),
\]
assume \(P(\nu<\infty)=1\), and assume \(\overline\Pi(n)>0\) at every finite \(n\). We set \(\overline\Pi(\infty)=0\), as required by Theorem~\ref{thm:weighted-universality}. A finite-support prior can instead be handled by the censoring argument in \Cref{rem:finite-horizon-pfa}.
Under the usual coupling in which observations before \(\nu\) have the no-change law,
\begin{equation}\label{eq:bayes-pfa-weight}
P(T<\nu)=\E_\infty[\overline\Pi(T)].
\end{equation}
Consequently, Bayesian false-alarm control is represented by an ordinary e-process crossing the boundary \(\overline\Pi(n)/\alpha\). For a geometric prior with hazard \(\rho\), this boundary is \((1-\rho)^n/\alpha\). It decreases with time because a late alarm is less likely to precede the random change. Weighted false-alarm probabilities of this form are classical in Bayesian quickest detection; see, for example, \citet{shiryaev1963,tartakovsky2014}.
\end{example}

\begin{remark}[Finite-horizon PFA]\label{rem:finite-horizon-pfa}
The loss \(w(n)=\1\{n\leq H\}\) vanishes at finite times after \(H\), so it falls outside the strictly positive convention used in Theorem~\ref{thm:weighted-universality}. The criterion is nevertheless an immediate PFA problem after censoring. Define
\[
T^{[H]}:=
\begin{cases}
T,&T\leq H,\\
\infty,&T>H.
\end{cases}
\]
Then \(P(T\leq H)=P(T^{[H]}<\infty)\), and the canonical e-process is
\[
E_n=\frac1\alpha\1\{T\leq n\wedge H\}.
\]
Its crossing time is \(T\) on \(\{T\leq H\}\) and infinity otherwise.
\end{remark}

\section{Deterministic tail envelopes and marginal e-certificates}\label{app:tail-envelope}

A deterministic family of false-alarm bounds has the form
\begin{equation}\label{eq:general-tail-envelope}
P(T\leq n)\leq g(n),\qquad n\geq1,
\end{equation}
where \(g:\N\to(0,1]\) is nondecreasing. This criterion is stronger than no restriction at all on early false alarms but weaker than optional-horizon validity: the horizon in \eqref{eq:general-tail-envelope} is fixed in advance.

\begin{definition}[Tail e-envelope]\label{def:tail-e-envelope}
A \emph{\(g\)-tail e-envelope} is a nonnegative adapted sequence of e-values \(E=(E_n)_{n\geq1}\), meaning that
\begin{equation}\label{eq:marginal-evalues}
\sup_{P\in\cP}\E_P[E_n]\leq1
\qquad\text{for every deterministic }n,
\end{equation}
such that the calibrated certificate
\begin{equation}\label{eq:absorbing-certificate}
A_n:=g(n)E_n
\end{equation}
 is pathwise nondecreasing. Its alarm time is
\begin{equation}\label{eq:tail-envelope-alarm}
T_g(E):=\inf\{n\geq1:A_n\geq1\}.
\end{equation}
\end{definition}

Thus each \(E_n\) is an e-value at the fixed time \(n\), but the sequence need not be an e-process. Monotonicity is imposed on the calibrated certificate \(A_n\), not on \(E_n\) itself.

\begin{theorem}[Universality for deterministic tail envelopes]\label{thm:tail-envelope-universality}
Fix a nondecreasing \(g:\N\to(0,1]\) and a stopping time \(T\). The following are equivalent.
\begin{enumerate}[label=\textup{(\roman*)},leftmargin=2.4em]
\item \(P(T\leq n)\leq g(n)\) for every \(P\in\cP\) and every \(n\geq1\).
\item There exists a \(g\)-tail e-envelope \(E\) such that \(T=T_g(E)\).
\end{enumerate}
Under \textup{(i)}, the canonical envelope is
\begin{equation}\label{eq:canonical-tail-envelope}
E_n^{T,g}:=\frac{\1\{T\leq n\}}{g(n)}.
\end{equation}
\end{theorem}

\begin{proof}
Assume \textup{(ii)}. Since \(A\) is nondecreasing,
\[
\{T_g(E)\leq n\}\subseteq\{A_n\geq1\}.
\]
Markov's inequality and \eqref{eq:marginal-evalues} imply
\[
P(T_g(E)\leq n)
\leq \E_P[A_n]
=g(n)\E_P[E_n]
\leq g(n).
\]

Conversely, assume \textup{(i)} and define \(E^{T,g}\) by \eqref{eq:canonical-tail-envelope}. Then \(\E_P[E_n^{T,g}]\leq1\), while
\[
g(n)E_n^{T,g}=\1\{T\leq n\}
\]
is nondecreasing. It jumps to one exactly at \(T\), proving the representation.
\end{proof}

\begin{example}[The linear tail envelope need not control ARL]\label{ex:uniform-tail-envelope}
Let \(b\geq2\) be an integer, reveal \(U\sim\operatorname{Unif}\{1,\ldots,b\}\) at time one, and set \(T=U\). Then
\[
P(T\leq n)=\min\left\{\frac nb,1\right\},
\]
so the linear deterministic tail bound holds with equality. Nevertheless,
\[
\E[T]=\frac{b+1}{2}<b.
\]
Thus deterministic linear-tail control does not imply ARL control. Conversely, Example~\ref{ex:early-lottery} shows that ARL control does not imply the linear tail bound. The two criteria are generally incomparable.

For the canonical envelope with \(g(n)=\min\{n/b,1\}\), evaluation at the data-dependent time \(T\) gives
\[
E_T^{T,g}=\frac{b}{T},
\qquad
\E[E_T^{T,g}]=\sum_{t=1}^b\frac1t>1.
\]
Hence the marginal e-values are not an e-process. This explicit optional-stopping failure is exactly what distinguishes deterministic-horizon control from optional-horizon control.
\end{example}

Quantile run-length requirements are special cases of \eqref{eq:general-tail-envelope}; one chooses \(g\) to encode the desired bounds at selected horizons and fills the gaps by a nondecreasing envelope. Such distribution-sensitive alternatives to ARL have been advocated in the change-detection literature, particularly when a large mean can conceal substantial early false-alarm risk; see \citet{mei2008}.

\section{Conditional hazards and e-hazard sequences}\label{app:e-hazard}

The most local false-alarm criterion controls the chance of alarming at the next observation, conditional on the complete current history. Let \(\alpha_n\in(0,1]\) be \(\cF_{n-1}\)-measurable. A stopping time \(T\) has \emph{historywise hazard control} if
\begin{equation}\label{eq:historywise-hazard}
P(T=n\mid\cF_{n-1})
\leq
\alpha_n\1\{T\geq n\}
\qquad P\text{-almost surely}
\end{equation}
for every \(P\in\cP\) and every \(n\geq1\).

\begin{definition}[E-hazard sequence]\label{def:e-hazard}
An \emph{e-hazard sequence} is a nonnegative adapted sequence \((E_n)_{n\geq1}\) satisfying
\begin{equation}\label{eq:conditional-evalue}
\E_P[E_n\mid\cF_{n-1}]\leq1
\qquad P\text{-almost surely}
\end{equation}
for every \(P\in\cP\) and every \(n\). Given a predictable boundary sequence \(1/\alpha_n\), define
\begin{equation}\label{eq:e-hazard-alarm}
T_\alpha(E):=\inf\left\{n\geq1:E_n\geq\frac1{\alpha_n}\right\}.
\end{equation}
\end{definition}

Each \(E_n\) is a one-step conditional e-value. Conditional e-values are standard ingredients in e-process constructions; see \citet{ramdaswang2025}. Here the terms are thresholded separately rather than multiplied.

\begin{theorem}[Universality of e-hazard sequences]\label{thm:e-hazard-universality}
Fix a predictable sequence \((\alpha_n)\) with values in \((0,1]\), and a stopping time \(T\). The following are equivalent.
\begin{enumerate}[label=\textup{(\roman*)},leftmargin=2.4em]
\item \(T\) satisfies historywise hazard control \eqref{eq:historywise-hazard}.
\item There exists an e-hazard sequence \(E\) such that \(T=T_\alpha(E)\).
\end{enumerate}
Under \textup{(i)}, the canonical sequence is
\begin{equation}\label{eq:canonical-e-hazard}
E_n^T:=\frac{\1\{T=n\}}{\alpha_n}.
\end{equation}
\end{theorem}

\begin{proof}
Suppose \textup{(ii)} holds. The event \(\{T_\alpha(E)\geq n\}\) belongs to \(\cF_{n-1}\), and conditional Markov inequality gives
\begin{align*}
P(T_\alpha(E)=n\mid\cF_{n-1})
&=\1\{T_\alpha(E)\geq n\}
  P\left(E_n\geq\frac1{\alpha_n}\,\middle|\,\cF_{n-1}\right)\\
&\leq \alpha_n\1\{T_\alpha(E)\geq n\}
  \E_P[E_n\mid\cF_{n-1}]\\
&\leq\alpha_n\1\{T_\alpha(E)\geq n\}.
\end{align*}

Conversely, assume \textup{(i)} and define \(E^T\) by \eqref{eq:canonical-e-hazard}. Then
\[
\E_P[E_n^T\mid\cF_{n-1}]
=\frac{P(T=n\mid\cF_{n-1})}{\alpha_n}
\leq\1\{T\geq n\}
\leq1.
\]
Before \(T\), the sequence is zero; at \(T\), it equals the boundary \(1/\alpha_T\). Hence its first individual crossing is exactly \(T\).
\end{proof}

\begin{remark}[Survival-conditioned hazard]\label{rem:at-risk-hazard}
The weaker criterion
\[
P(T=n\mid T\geq n)\leq\alpha_n
\]
for deterministic \(\alpha_n\) has an analogous representation by \emph{at-risk e-values}, defined through
\[
\E_P[E_n\mid T\geq n]\leq1.
\]
The same canonical variable \(\alpha_n^{-1}\1\{T=n\}\) proves the converse. This version averages over histories among paths still under surveillance, whereas \eqref{eq:historywise-hazard} protects every \(\cF_{n-1}\)-history.
\end{remark}

One-step hazard bounds integrate into an optional-clock guarantee.

\begin{proposition}[From local hazards to an optional budget]\label{prop:hazard-clock}
Suppose \(q_n\in[0,1]\) is predictable and
\begin{equation}\label{eq:q-hazard}
P(T=n\mid\cF_{n-1})\leq q_n\1\{T\geq n\}
\end{equation}
under every \(P\in\cP\). Define the cumulative hazard clock
\[
C_n:=\sum_{j=1}^n q_j.
\]
Then, for every stopping horizon \(\sigma\), with both sides interpreted in \([0,\infty]\),
\begin{equation}\label{eq:hazard-clock-bound}
P(T\leq\sigma)\leq\E_P[C_\sigma].
\end{equation}
Equivalently, \(\1\{T\leq n\}\) is a strong \(C\)-clocked e-process.
\end{proposition}

\begin{proof}
Since \(\{\sigma\geq n\}\in\cF_{n-1}\), Tonelli's theorem and \eqref{eq:q-hazard} give
\begin{align*}
P(T\leq\sigma)
&=\sum_{n=1}^\infty
  \E_P\bigl[\1\{\sigma\geq n\}\1\{T=n\}\bigr]\\
&\leq\sum_{n=1}^\infty
  \E_P\bigl[\1\{\sigma\geq n\}q_n\1\{T\geq n\}\bigr]\\
&\leq\sum_{n=1}^\infty
  \E_P\bigl[\1\{\sigma\geq n\}q_n\bigr]
=\E_P[C_\sigma].
\end{align*}
The final statement follows from Theorem~\ref{thm:clocked-strong-universality} with \(b=1\), or directly from the same calculation.
\end{proof}

Taking \(q_n\equiv1/b\) recovers Proposition~\ref{prop:hazard}. In that case the canonical one-step certificates are \(E_n=b\1\{T=n\}\), while their cumulative sum is
\[
\sum_{j=1}^nE_j=b\1\{T\leq n\},
\]
the canonical strong e-detector from Theorem~\ref{thm:strong-universality}. The local and cumulative e-concepts therefore fit together exactly.

\section{Local-window control and rolling e-fields}\label{app:local-window}

A one-step hazard can be too granular for practice. A natural compromise controls the probability of at least one false alarm during each upcoming window of length \(m\), conditional on the procedure still running at the start of that window. Local unconditional and conditional false-alarm classes of this kind have been studied as alternatives to ARL, particularly for non-i.i.d. observations; see \citet{pergamenchtchikov2018} and the discussion surrounding \citet{mei2008}.

Fix \(m\in\N\) and \(\alpha\in(0,1]\). We use the ratio form
\begin{equation}\label{eq:local-window-ratio}
P(k<T\leq k+m)\leq\alpha P(T>k),
\qquad k\geq0,
\end{equation}
which remains meaningful when \(P(T>k)=0\) and is equivalent to
\[
P(k<T\leq k+m\mid T>k)\leq\alpha
\]
whenever the conditioning event has positive probability.

The criterion is indexed by every possible restart time \(k\), so the natural representing object has two time indices.

\begin{definition}[At-risk rolling e-field]\label{def:rolling-e-field}
For every \(k\geq0\), let
\[
E^{(k)}=(E_n^{(k)})_{n=k}^{k+m}
\]
be a nonnegative adapted row with \(E_k^{(k)}=0\). Define the rolling-consensus alarm
\begin{equation}\label{eq:rolling-consensus}
T_{\mathcal E}:=
\inf\left\{n\geq1:
\min_{(n-m)\vee0\leq k\leq n-1}E_n^{(k)}
\geq\frac1\alpha
\right\}.
\end{equation}
The field is an \emph{at-risk rolling \(m\)-window e-field} if, for every \(P\in\cP\), every \(k\geq0\), and every stopping time \(\tau\) taking values in \(\{k,\ldots,k+m\}\),
\begin{equation}\label{eq:at-risk-row-validity}
\E_P\bigl[E_\tau^{(k)}\1\{T_{\mathcal E}>k\}\bigr]
\leq P(T_{\mathcal E}>k).
\end{equation}
\end{definition}

When \(P(T_{\mathcal E}>k)>0\), \eqref{eq:at-risk-row-validity} says that the \(k\)-th row is an e-process under the survival-conditioned null law \(P(\cdot\mid T_{\mathcal E}>k)\), restricted to the next \(m\) observations. The definition is mathematically well posed: the field determines \(T_{\mathcal E}\) pathwise before validity is assessed. It is nevertheless self-referential as a construction criterion, because the validity of one row depends on the alarm generated by the entire field. Accordingly, \Cref{thm:rolling-universality} is best viewed as an exact representation theorem, not as a recipe whose rows can be verified independently. The stronger historywise formulation in \Cref{prop:historywise-window} has genuinely row-wise conditional validity.

\begin{theorem}[Universality for local conditional false-alarm control]\label{thm:rolling-universality}
Fix \(m\) and \(\alpha\), and let \(T\) be a stopping time. The following are equivalent.
\begin{enumerate}[label=\textup{(\roman*)},leftmargin=2.4em]
\item \(T\) satisfies \eqref{eq:local-window-ratio} for every \(P\in\cP\) and every \(k\geq0\).
\item There exists an at-risk rolling \(m\)-window e-field \(\mathcal E\) such that \(T=T_{\mathcal E}\).
\end{enumerate}
Under \textup{(i)}, a canonical field is
\begin{equation}\label{eq:canonical-window-field}
E_n^{(k)}:=\frac1\alpha\1\{k<T\leq n\},
\qquad k\leq n\leq k+m.
\end{equation}
\end{theorem}

\begin{proof}
Assume \textup{(ii)} and write \(T=T_{\mathcal E}\). Fix \(k\), and let
\[
\tau_k:=
\left(\inf\left\{n\in\{k+1,\ldots,k+m\}:E_n^{(k)}\geq\frac1\alpha\right\}\right)\wedge(k+m),
\]
with the usual convention that the infimum of the empty set is infinity. If \(k<T\leq k+m\), then the \(k\)-th row is active at time \(T\), and the consensus condition forces \(E_T^{(k)}\geq1/\alpha\). Hence that row crosses by \(\tau_k\), and
\[
\frac1\alpha P(k<T\leq k+m)
\leq
\E_P\bigl[E_{\tau_k}^{(k)}\1\{T>k\}\bigr]
\leq P(T>k).
\]
This is \eqref{eq:local-window-ratio}.

Conversely, assume \textup{(i)} and define the field by \eqref{eq:canonical-window-field}. For any row-valued stopping time \(\tau\),
\begin{align*}
\E_P\bigl[E_\tau^{(k)}\1\{T>k\}\bigr]
&=\frac1\alpha P(k<T\leq\tau)\\
&\leq\frac1\alpha P(k<T\leq k+m)
\leq P(T>k),
\end{align*}
so the field is at-risk valid. Before \(T\), every active row is zero. At time \(T\), every active index \(k\in\{(T-m)\vee0,\ldots,T-1\}\) satisfies \(k<T\leq k+m\), so every active row equals \(1/\alpha\). The rolling-consensus crossing is therefore exactly \(T\).
\end{proof}

The minimum in \eqref{eq:rolling-consensus} expresses a logical consensus: an alarm lying in several overlapping windows must carry a certificate for every one of those windows. The canonical field shows that this loses no generality as a representation theorem. It may nevertheless be conservative and inconvenient as a design principle, especially because survival-conditioned row validity cannot be checked independently.

There is a stronger historywise version.

\begin{proposition}[Historywise local-window universality]\label{prop:historywise-window}
Replace \eqref{eq:local-window-ratio} by
\begin{equation}\label{eq:historywise-window}
P(k<T\leq k+m\mid\cF_k)
\leq\alpha\1\{T>k\}
\qquad P\text{-almost surely}.
\end{equation}
Then the same rolling-consensus representation is exact if row validity is strengthened to
\begin{equation}\label{eq:conditional-row-validity}
\E_P[E_\tau^{(k)}\mid\cF_k]\leq1
\qquad P\text{-almost surely}
\end{equation}
for every row-valued stopping time \(\tau\). The canonical field remains \eqref{eq:canonical-window-field}.
\end{proposition}

\begin{proof}
Conditional Markov inequality applied to the first crossing of the \(k\)-th row proves soundness. Conversely,
\[
\E_P[E_\tau^{(k)}\mid\cF_k]
=\frac1\alpha P(k<T\leq\tau\mid\cF_k)
\leq\frac1\alpha P(k<T\leq k+m\mid\cF_k)
\leq\1\{T>k\}
\leq1.
\]
The pathwise recovery argument is unchanged.
\end{proof}

For \(m=1\), Proposition~\ref{prop:historywise-window} reduces to the e-hazard representation in Theorem~\ref{thm:e-hazard-universality}. The standard survival-conditioned version in Theorem~\ref{thm:rolling-universality} reduces to the at-risk hazard notion in Remark~\ref{rem:at-risk-hazard}.

An unconditional local-window criterion,
\[
P(k<T\leq k+m)\leq\alpha,
\]
has the same representation with ordinary finite-horizon e-process rows, obtained by dropping \(\1\{T>k\}\) and \(P(T>k)\) from \eqref{eq:at-risk-row-validity}. Thus unconditional, survival-conditioned, and historywise local control correspond to three increasingly strong notions of row-wise e-validity.

\section{Residual criteria and restartable detector fields}\label{app:residual-arl}

A global ARL guarantee is evaluated only at startup. Residual criteria ask whether comparable protection remains after the procedure has survived to an arbitrary age. They naturally produce restart-indexed fields rather than one scalar process. There are two different versions, mirroring the weak--strong distinction in the main text.

\subsection{Residual ARL and weak rows}

The survival-conditioned residual ARL criterion is
\begin{equation}\label{eq:residual-arl}
\E_P[T-k\mid T>k]\geq b
\end{equation}
for every \(P\in\cP\) and every \(k\) with \(P(T>k)>0\).

For each such \(k\), let
\begin{equation}\label{eq:conditional-null-class}
\cP_k^T:=
\left\{P(\,\cdot\mid T>k):P\in\cP,\ P(T>k)>0\right\},
\end{equation}
viewed on the shifted filtration \((\cF_{k+j})_{j\geq0}\).

\begin{definition}[Restartable weak e-detector field]\label{def:restartable-field}
A \emph{restartable weak e-detector field at scale \(b\)} for \(T\) is a family
\[
\mathcal M=\{M^{(k)}=(M_j^{(k)})_{j\geq0}:k\geq0\}
\]
such that, for each \(k\), the shifted process \(M^{(k)}\) is a weak \(\cP_k^T\)-e-detector and
\begin{equation}\label{eq:restart-recovery}
T-k=T_b(M^{(k)})
\qquad\text{under every law in }\cP_k^T.
\end{equation}
\end{definition}

\begin{theorem}[Universality for residual ARL]\label{thm:residual-arl-universality}
Fix \(b>0\) and an alarm time \(T\). The following are equivalent.
\begin{enumerate}[label=\textup{(\roman*)},leftmargin=2.4em]
\item \(T\) satisfies residual ARL control \eqref{eq:residual-arl} for every null law and every admissible restart time \(k\).
\item \(T\) admits a restartable weak e-detector field at scale \(b\).
\end{enumerate}
Under \textup{(i)}, the canonical row started at \(k\) is
\begin{equation}\label{eq:canonical-restart-field}
M_j^{(k)}:=b\1\{k<T\leq k+j\},
\qquad j\geq0.
\end{equation}
\end{theorem}

\begin{proof}
If \textup{(ii)} holds, apply \Cref{thm:weak-soundness} under each conditional null class \(\cP_k^T\). The level-\(b\) crossing has conditional mean at least \(b\), which is exactly \eqref{eq:residual-arl}.

Conversely, suppose \textup{(i)} holds and define \(M^{(k)}\) by \eqref{eq:canonical-restart-field}. Under any law in \(\cP_k^T\), the process jumps from zero to \(b\) after exactly \(T-k\) shifted observations, so \eqref{eq:restart-recovery} holds. At every level \(0<c\leq b\), its crossing time is \(T-k\), and the killed stopped value satisfies
\[
\E\bigl[(M_{T-k}^{(k)})^\dagger\bigr]
=bP(T<\infty\mid T>k)
\leq b
\leq\E[T-k\mid T>k].
\]
At levels \(c>b\), it never crosses. Hence every row is a weak e-detector under the corresponding survival-conditioned null class.
\end{proof}

\begin{example}[ARL does not imply residual ARL]\label{ex:arl-not-residual}
For the early-alarm lottery in \Cref{ex:early-lottery}, \(T\) equals one or seven with equal probability and has \(\E[T]=4\). Conditional on survival to time six, however, \(T=7\) with probability one, so
\[
\E[T-6\mid T>6]=1.
\]
Thus the rule has a weak e-detector representation at scale four, but it has no restartable weak e-detector field at that scale.
\end{example}

\subsection{Historywise optional residual control and strong rows}

A strong restart condition allows the future horizon to depend on the observations after the restart and requires validity conditional on the entire current history. Call a random variable \(\tau\in\Nzero\cup\{\infty\}\) a \emph{shifted stopping time after \(k\)} if \(\{\tau\leq j\}\in\cF_{k+j}\) for every \(j\geq0\). Consider
\begin{equation}\label{eq:historywise-residual-optional}
P(k<T\leq k+\tau\mid\cF_k)
\leq
\frac{\1\{T>k\}}{b}\E_P[\tau\mid\cF_k]
\end{equation}
for every \(P\in\cP\), every \(k\geq0\), and every bounded shifted stopping time \(\tau\). Bounded horizons suffice by conditional monotone convergence.

\begin{definition}[Restartable strong e-detector field]\label{def:restartable-strong-field}
A \emph{restartable strong e-detector field at scale \(b\)} for \(T\) is a family \(\{M^{(k)}:k\geq0\}\) such that, under every survival-conditioned law in \(\cP_k^T\), the row \(M^{(k)}\) is a strong e-detector on the shifted filtration and its level-\(b\) crossing is \(T-k\).
\end{definition}

\begin{theorem}[Strong restartability is conditional-hazard control]\label{thm:strong-restart-hazard}
Fix \(b>0\). The following are equivalent.
\begin{enumerate}[label=\textup{(\roman*)},leftmargin=2.4em]
\item For every \(P\in\cP\) and every \(n\geq1\),
\begin{equation}\label{eq:restart-historywise-hazard}
P(T=n\mid\cF_{n-1})
\leq b^{-1}\1\{T\geq n\}
\qquad P\text{-almost surely}.
\end{equation}
\item The historywise optional residual inequality \eqref{eq:historywise-residual-optional} holds.
\item \(T\) admits a restartable strong e-detector field at scale \(b\).
\end{enumerate}
Under these conditions, the canonical strong row is again
\[
M_j^{(k)}=b\1\{k<T\leq k+j\}.
\]
\end{theorem}

\begin{proof}
Assume \textup{(i)}. For a bounded shifted stopping time \(\tau\), the event \(\{\tau\geq j\}\) belongs to \(\cF_{k+j-1}\). Conditional Tonelli and the tower property give
\begin{align*}
P(k<T\leq k+\tau\mid\cF_k)
&=\sum_{j\geq1}\E_P\!\left[\1\{\tau\geq j\}\1\{T=k+j\}\mid\cF_k\right]\\
&\leq \frac1b\sum_{j\geq1}\E_P\!\left[\1\{\tau\geq j\}\1\{T\geq k+j\}\mid\cF_k\right]\\
&\leq \frac{\1\{T>k\}}b
   \sum_{j\geq1}P(\tau\geq j\mid\cF_k)\\
&=\frac{\1\{T>k\}}b\E_P[\tau\mid\cF_k],
\end{align*}
which is \textup{(ii)}. Conversely, \textup{(ii)} with \(\tau\equiv1\) yields \textup{(i)}.

Under \textup{(ii)}, fix \(k\) and \(P\) with \(P(T>k)>0\), integrate \eqref{eq:historywise-residual-optional}, and divide by \(P(T>k)\). For every bounded shifted stopping time \(\tau\),
\[
bP(T-k\leq\tau\mid T>k)
\leq \E_P[\tau\mid T>k].
\]
Therefore the canonical row is strong under \(P(\cdot\mid T>k)\) and crosses at \(T-k\), proving \textup{(iii)}.

Finally, assume \textup{(iii)}. Fix \(k\), a null law with \(P(T>k)>0\), and an event \(A\in\cF_k\). Under \(Q=P(\cdot\mid T>k)\), the shifted horizon \(\tau_A:=\1_A\) is a stopping time. Optional-horizon validity of the row gives
\[
Q(A\cap\{T=k+1\})
=Q(T-k\leq\tau_A)
\leq \frac{Q(A)}b.
\]
Since this holds for every \(A\in\cF_k\), it implies \eqref{eq:restart-historywise-hazard} at time \(k+1\). Thus \textup{(iii)} implies \textup{(i)}.
\end{proof}

The theorem shows that, when the full shifted filtration is retained, a uniformly strong restart property is not a genuinely intermediate criterion: it is exactly the historywise hazard condition from \Cref{prop:hazard}. It does imply global strong validity by taking \(k=0\), and the implication is strict. For example, the deterministic run length \(T\equiv b\) in \Cref{ex:deterministic-strong} is strongly representable at scale \(b>1\), but after survival to time \(b-1\) the residual alarm time is one with probability one, violating \eqref{eq:historywise-residual-optional}. By contrast, the weak residual criterion only controls a conditional mean and does not collapse to a one-step hazard requirement.

\section{Repeated alarms and counting e-detectors}\label{app:counting}

When monitoring restarts after an alarm, the first run length no longer captures the full type-I behavior. Repeated or multi-cyclic operation is classical in statistical process control and in stationary formulations of Shiryaev--Roberts monitoring; see, for example, \citet{pollaktartakovsky2009}. Let \(N_n\) denote the number of false alarms declared by time \(n\), with \(N_0=0\), and assume \(N\) is adapted, integer-valued, and pathwise nondecreasing.

A natural optional count-rate criterion is
\begin{equation}\label{eq:optional-count-rate}
\E_P[N_\tau]\leq\lambda\E_P[\tau]
\end{equation}
for every \(P\in\cP\) and every integrable stopping time \(\tau\), where \(\lambda>0\) is the allowed expected false-alarm intensity.

\begin{proposition}[Counting e-detectors]\label{prop:counting-edetector}
The optional count-rate condition \eqref{eq:optional-count-rate} holds if and only if
\begin{equation}\label{eq:normalized-count-detector}
M_n:=\frac{N_n}{\lambda}
\end{equation}
is a strong \(\cP\)-e-detector.
\end{proposition}

\begin{proof}
Substituting \eqref{eq:normalized-count-detector} into the defining inequality of a strong e-detector gives exactly \eqref{eq:optional-count-rate}.
\end{proof}

Although elementary, this equivalence has useful consequences. Let
\begin{equation}\label{eq:jth-alarm}
S_j:=\inf\{n\geq1:N_n\geq j\}
\end{equation}
be the time of the \(j\)-th false alarm. Then \(S_j=T_{j/\lambda}(M)\), and the strong e-detector results give
\begin{align}
\E_P[S_j]&\geq\frac{j}{\lambda},\label{eq:jth-alarm-mean}\\
P(S_j\leq\sigma)&\leq\frac{\lambda\E_P[\sigma]}{j}
\label{eq:jth-alarm-optional}
\end{align}
for every integrable stopping horizon \(\sigma\). Thus a single counting e-detector controls every successive false-alarm time.

A predictable intensity bound gives a more local representation. Write \(\Delta N_n=N_n-N_{n-1}\), let \(\lambda_n>0\) be predictable, and define
\[
C_n:=\sum_{j=1}^n\lambda_j.
\]

\begin{proposition}[Conditional e-values for false-alarm increments]\label{prop:count-increments}
The following are equivalent for each \(n\):
\begin{enumerate}[label=\textup{(\roman*)},leftmargin=2.4em]
\item
\begin{equation}\label{eq:predictable-count-intensity}
\E_P[\Delta N_n\mid\cF_{n-1}]\leq\lambda_n
\qquad P\text{-almost surely for every }P\in\cP.
\end{equation}
\item \(E_n:=\Delta N_n/\lambda_n\) is a conditional e-value, meaning
\[
\E_P[E_n\mid\cF_{n-1}]\leq1.
\]
\end{enumerate}
Moreover, if \eqref{eq:predictable-count-intensity} holds for every \(n\), then \(N\) is a strong \(C\)-clocked e-process:
\begin{equation}\label{eq:count-clocked}
\E_P[N_\tau]\leq\E_P[C_\tau]
\end{equation}
for every bounded stopping time \(\tau\), and by monotone convergence for every stopping time for which the right side is finite.
\end{proposition}

\begin{proof}
The equivalence is immediate by division by the positive predictable variable \(\lambda_n\). For the cumulative statement, \(\{\tau\geq n\}\in\cF_{n-1}\), so
\begin{align*}
\E_P[N_\tau]
&=\sum_{n=1}^\infty
  \E_P\bigl[\1\{\tau\geq n\}\Delta N_n\bigr]\\
&\leq\sum_{n=1}^\infty
  \E_P\bigl[\1\{\tau\geq n\}\lambda_n\bigr]
=\E_P[C_\tau].
\end{align*}
\end{proof}

\begin{example}[Conditionally Bernoulli false alarms]\label{ex:bernoulli-counts}
Let \(\lambda\in(0,1)\). Suppose \(\Delta N_n\in\{0,1\}\) and
\[
P(\Delta N_n=1\mid\cF_{n-1})\leq\lambda
\]
at every time. Then \(\Delta N_n/\lambda\) is a conditional e-value, \(N_n/\lambda\) is a strong e-detector, and the \(j\)-th false alarm has mean at least \(j/\lambda\). Under equality with conditionally independent increments, \(S_j\) has the negative-binomial waiting-time law with mean exactly \(j/\lambda\).
\end{example}

The deterministic criterion \(\E[N_n]\leq\lambda n\) for every \(n\) corresponds only to a sequence of fixed-time e-values \(N_n/(\lambda n)\); it does not by itself imply optional count-rate control.

\begin{example}[Fixed-time count control is not optional]\label{ex:count-fixed-not-optional}
Let \(b\geq2\) be an integer, set \(\lambda=1/b\), reveal \(U\sim\operatorname{Unif}\{1,\ldots,b\}\) at time one, and define
\[
N_n:=\1\{n\geq U\}.
\]
Then
\[
\E[N_n]=\min\{n/b,1\}\leq\lambda n
\]
for every deterministic \(n\). At the stopping time \(\tau=U\), however,
\[
\E[N_\tau]=1
>\frac{b+1}{2b}
=\lambda\E[\tau].
\]
Thus deterministic expected-count bounds do not imply optional count-rate control.
\end{example}

The still weaker asymptotic condition
\[
\limsup_{n\to\infty}\frac{\E[N_n]}{n}\leq\lambda
\]
can be written as \(\limsup_n\E[M_n]/n\leq1\) for \(M=N/\lambda\), but this is merely an asymptotic normalization, not an anytime-valid e-concept. For repeated alarms, optional or predictable intensity control is therefore the most natural analogue of strong e-detector validity.

\section{Simultaneous representation of a threshold family}\label{app:simultaneous}

The weak universality theorem represents one alarm time at one target level. A stronger question asks when a family \(\{T_b:b>0\}\) can be realized as \emph{all} threshold crossings of a single weak e-detector.

Call the family \emph{threshold coherent} if, pathwise,
\begin{align}
 b<c&\implies T_b\leq T_c,\label{eq:family-monotone}\\
 T_b&=\sup_{q\in\mathbb Q_+,\,q<b}T_q
 \qquad(b>0).\label{eq:family-left-cont}
\end{align}
The second condition is the left-continuity forced by crossings of a real-valued process. (In \Cref{ex:family-insufficient}, the ceiling function causes no problem because \(b\mapsto\lceil b\rceil\) is left-continuous.) Define the canonical running level
\begin{equation}\label{eq:canonical-family-level}
A_n:=\sup\{q\in\mathbb Q_+:T_q\leq n\},
\qquad \sup\varnothing:=0.
\end{equation}
It is adapted and pathwise nondecreasing. We call the family \emph{locally bounded} if \(A_n<\infty\) pathwise for every finite \(n\); this is necessary for representation by a finite-valued process.

\begin{theorem}[Simultaneous-family characterization]\label{thm:simultaneous-family}
A family \(\{T_b:b>0\}\) is representable as
\[
T_b=T_b(M)\qquad\text{for every }b>0
\]
by one finite-valued weak \(\cP\)-e-detector \(M\) if and only if it is threshold coherent, locally bounded, and, for every \(P\in\cP\) and every \(b>0\),
\begin{equation}\label{eq:family-overshoot-condition}
\E_P[A_{T_b}^\dagger]\leq\E_P[T_b],
\end{equation}
where \(A\) is defined by \eqref{eq:canonical-family-level}. When these conditions hold, \(A\) itself is a representing weak e-detector.
\end{theorem}

\begin{proof}
Assume first that the family is threshold coherent. The definitions imply
\[
A_n\geq b\quad\Longleftrightarrow\quad T_b\leq n.
\]
The forward implication uses \eqref{eq:family-left-cont} when the supremum in \eqref{eq:canonical-family-level} reaches \(b\) only through rationals below \(b\). Hence \(T_b(A)=T_b\) for every \(b\). Condition \eqref{eq:family-overshoot-condition} is exactly weak validity of \(A\) at every level.

Conversely, suppose a weak detector \(M\) represents the family. Threshold crossings of any process satisfy \eqref{eq:family-monotone}--\eqref{eq:family-left-cont}. Replace \(M\) by its running maximum \(\overline M_n=\max_{j\leq n}M_j\). This preserves every crossing time and every value at the first crossing, hence preserves weak validity. The family determines this running maximum uniquely through
\[
\overline M_n
=\sup\{b:T_b(M)\leq n\}
=A_n.
\]
Therefore \(A\) must satisfy \eqref{eq:family-overshoot-condition}.
\end{proof}

The extra condition is an overshoot constraint, and scale-wise ARL bounds do not imply it.

\begin{example}[Finite ARL at every scale is insufficient]\label{ex:family-insufficient}
Fix \(p\in(0,1)\) and \(H>1\), let \(Y\sim\operatorname{Bernoulli}(p)\), and reveal \(Y\) at time one. For \(0<b\leq H\), define
\[
T_b=
\begin{cases}
1, &0<b\leq1,\\
1, &1<b\leq H\text{ and }Y=1,\\
\left\lceil\dfrac{b-p}{1-p}\right\rceil,
&1<b\leq H\text{ and }Y=0.
\end{cases}
\]
Choose a constant \(K\geq1\) large enough that \(\lceil Kb\rceil\geq\lceil(H-p)/(1-p)\rceil\) for every \(b>H\), and set \(T_b=\lceil Kb\rceil\) deterministically for \(b>H\). The resulting family is threshold coherent and every crossing time is finite. Moreover,
\[
\E[T_b]
=p+(1-p)\left\lceil\frac{b-p}{1-p}\right\rceil
\geq b,
\qquad 1<b\leq H,
\]
while \(\E[T_b]=1\geq b\) for \(b\leq1\) and \(\E[T_b]=\lceil Kb\rceil\geq b\) for \(b>H\). Thus the scale-wise ARL inequality holds with finite expectation at every scale.

At time one, however, the canonical running level is
\[
A_1=
\begin{cases}
H,&Y=1,\\
1,&Y=0.
\end{cases}
\]
Since \(T_1=1\),
\[
\E[A_{T_1}]
=pH+(1-p)
>1
=\E[T_1].
\]
The overshoot condition in \Cref{thm:simultaneous-family} fails. Hence no single weak e-detector can realize the entire family, even though every \(T_b\) is almost surely finite and separately has a weak representation at its own scale.
\end{example}

\section{Conjunctions of ARL and tail criteria}\label{app:conjunction}

A procedure may be required to have both a large mean run length and a prescribed deterministic lower-tail envelope. These two conditions can be represented by one process, but validity must be imposed in two different ways.

For a process \(M\), write \(\overline M_n=\max_{1\leq j\leq n}M_j\).

\begin{theorem}[Universality for an ARL--tail conjunction]\label{thm:conjunction}
Fix \(b>0\), a nondecreasing \(g:\N\to(0,1]\), and an alarm time \(T\). The following are equivalent.
\begin{enumerate}[label=\textup{(\roman*)},leftmargin=2.4em]
\item For every \(P\in\cP\),
\begin{equation}\label{eq:conjunction-criterion}
\E_P[T]\geq b
\quad\text{and}\quad
P(T\leq n)\leq g(n)\ \text{ for every }n\geq1.
\end{equation}
\item There exists a weak \(\cP\)-e-detector \(M\) such that \(T=T_b(M)\) and
\begin{equation}\label{eq:conjunction-running-budget}
\E_P[\overline M_n]\leq b g(n)
\qquad\text{for every }P\in\cP\text{ and }n\geq1.
\end{equation}
\end{enumerate}
Under \textup{(i)}, the canonical witness is \(M_n=b\1\{T\leq n\}\).
\end{theorem}

\begin{proof}
Under \textup{(ii)}, weak soundness gives \(\E_P[T]\geq b\). Moreover,
\[
\{T\leq n\}=\{T_b(M)\leq n\}
\subseteq\{\overline M_n\geq b\},
\]
so Markov's inequality and \eqref{eq:conjunction-running-budget} give \(P(T\leq n)\leq g(n)\).

Conversely, define \(M_n=b\1\{T\leq n\}\). The ARL part of \eqref{eq:conjunction-criterion} makes \(M\) weak by \Cref{thm:weak-universality}. Since \(M\) is nondecreasing,
\[
\E_P[\overline M_n]
=bP(T\leq n)
\leq bg(n).
\]
\end{proof}

The conjunction does not require taking the maximum of two separate indicator certificates: the same all-or-nothing process works, but it is tested both at its own crossings and at deterministic horizons. For \(g(n)=\min\{n/b,1\}\), this conjunction remains weaker than optional-horizon control; \Cref{ex:det-not-optional} satisfies both ARL and the linear deterministic envelope but fails the adaptive-horizon inequality.

\section{Synthesis: horizons, clocks, windows, and restarts}\label{app:synthesis}
\enlargethispage{2\baselineskip}

The main clock theorem and the preceding appendices suggest a general taxonomy. An error metric specifies both a \emph{budget} and a class of \emph{horizons} at which that budget must remain valid. Fixed-time validity yields marginal e-values; validity at all stopping times yields e-processes or strong clocked e-processes; validity only at a process's own crossings yields weak e-detectors; criteria indexed by restart times require fields; and simultaneous requirements impose several validity conditions on the same process.

\begin{table}[ht]
\centering
\footnotesize
\renewcommand{\arraystretch}{1.16}
\begin{tabularx}{\textwidth}{@{}>{\raggedright\arraybackslash}p{0.29\textwidth}>{\raggedright\arraybackslash}p{0.31\textwidth}X@{}}
\toprule
False-alarm criterion & Universal e-representation & Canonical certificate \\
\midrule
\(\E[w(T)]\leq\alpha\) & e-process with boundary \(w(n)/\alpha\) & \(\alpha^{-1}w(T)\1\{T\leq n\}\) \\
\(\E[C_T]\geq b\) & weak \(C\)-clocked e-process & \(b\1\{T\leq n\}\) \\
\(P(T\leq\sigma)\leq\E[C_\sigma]/b\) & strong \(C\)-clocked e-process & \(b\1\{T\leq n\}\) \\
\(P(T\leq n)\leq g(n)\) & \(g\)-tail e-envelope & \(g(n)^{-1}\1\{T\leq n\}\) \\
Historywise hazard \(\leq\alpha_n\) & e-hazard sequence & \(\alpha_n^{-1}\1\{T=n\}\) \\
Local \(m\)-window probability \(\leq\alpha\) & rolling at-risk e-field & \(\alpha^{-1}\1\{k<T\leq n\}\) \\
Residual ARL \(\geq b\) & restartable weak e-detector field & \(b\1\{k<T\leq k+j\}\) \\
Historywise optional residual control & restartable strong e-detector field; equivalently e-hazard control & \(b\1\{k<T\leq k+j\}\) \\
Optional expected count rate \(\leq\lambda\) & counting e-detector \(N/\lambda\) & normalized count process \\
One family \(\{T_b\}\) at every threshold & coherent weak detector plus overshoot control & canonical running level \(A_n\) \\
ARL and deterministic tail jointly & weak detector plus running-maximum budget & \(b\1\{T\leq n\}\) \\
\bottomrule
\end{tabularx}
\caption{Exact e-representations for alternative false-alarm criteria. The clock terminology is developed in the main text; the terms e-hazard, tail e-envelope, and rolling e-field are proposed in the appendices.}
\label{tab:extended-universality}
\end{table}

Several patterns are worth emphasizing.

First, the canonical representations remain all-or-nothing. Their purpose is to establish completeness of an e-language for a validity criterion, not to provide a powerful detection statistic. Practical procedures should generally seek smoother evidence accumulation, favorable post-change growth, and efficient delay.

Second, the same indicator \(b\1\{T\leq n\}\) can represent distinct criteria because the validity requirement is imposed at different horizon classes. Own crossings encode an expected-clock lower bound; deterministic times encode a tail envelope; all stopping times encode optional-clock control. Under the minimal alarm filtration, the intermediate family of self-censoring horizons \(T\wedge m\) already characterizes optional validity through the truncated-mean inequalities in \Cref{thm:run-length-law}. The process formula alone does not identify the error metric---the quantification over horizons and the available filtration do.

Third, local criteria are intrinsically restart-indexed. A scalar process can encode a global PFA, ARL, or optional-clock constraint, but a guarantee for every rolling window or every residual lifetime naturally produces a two-parameter family. Weak residual ARL yields genuinely weaker restartable rows. Requiring strong row validity after every history collapses exactly to the one-step conditional-hazard criterion, so the e-hazard sequence is the strong boundary case of the restartable viewpoint.

Finally, clocks separate the amount of budget from the passage of calendar time. The choices \(C_n\equiv1\), \(C_n=n\), and \(C_n=\sum_{j\leq n}a_j\) correspond, respectively, to a fixed testing budget, a linearly replenished detection budget, and an exposure-weighted budget. This clock perspective provides a common home for e-processes, e-detectors, and several new variants developed above.

\end{document}